\pdfoutput=1
\documentclass[microtype]{gtpart}
\usepackage{graphicx}
\usepackage[mathscr]{eucal}
\usepackage{amssymb} 
\usepackage[dvipsnames]{xcolor}
\usepackage{enumerate, cite, soul}
\usepackage{overpic}

\newcommand{\br}{\mathbb{R}}
\newcommand{\bc}{\mathbb C}
\newcommand{\bz}{\mathbb Z}
\newcommand{\bn}{\mathbb N}
\newcommand{\bq}{\mathbb Q}
\newcommand{\bh}{\mathbb H}
\newcommand{\bS}{\mathbb S}

\newcommand{\cE}{\mathcal E}

\newcommand{\al}{\alpha}
\newcommand{\be}{\beta}

\newcommand{\vp}{\varphi}

\newcommand{\ssm}{\smallsetminus}

\DeclareMathOperator{\arcsinh}{arcsinh}
\DeclareMathOperator{\mcg}{MCG}
\DeclareMathOperator{\pmcg}{PMCG}
\DeclareMathOperator{\isom}{Isom}
\DeclareMathOperator{\Ends}{\cE}

\DeclareMathOperator{\Homeo}{Homeo}
\DeclareMathOperator{\Diffeo}{Diffeo}
\DeclareMathOperator{\Aut}{Aut}

\newtheorem{Thm}{Theorem}[section]
\newtheorem{Thm*}{Theorem}
\newtheorem{Prop}[Thm]{Proposition}
\newtheorem{Lem}[Thm]{Lemma}
\newtheorem{Cor}[Thm]{Corollary}
\newtheorem{Cor*}[Thm*]{Corollary}

\newtheorem{problem}[Thm*]{Problem}
\newtheorem*{ThmA}{Theorem~\ref*{mainUncountable}}
\newtheorem*{ThmB}{Theorem~\ref*{main}}
\newtheorem*{ThmC}{Theorem~\ref*{thm:covering}}
\newtheorem*{ThmD}{Theorem~\ref*{thm:rigid}}

\theoremstyle{definition}
\newtheorem{Def}[Thm]{Definition}

\newtheorem{Rem}[Thm]{Remark}

\numberwithin{equation}{section}

\title{Isometry groups of infinite-genus hyperbolic surfaces}

\author{Tarik Aougab}
\address{Department of Mathematics, Haverford College, Haverford, PA 19104}
\email{taougab@haverford.edu}

\author{Priyam Patel}
\address{Department of Mathematics, University of Utah \\ Salt Lake City, UT 84112}
\email{patelp@math.utah.edu}

\author{Nicholas G. Vlamis}
\address{Department of Mathematics, CUNY Queens College \\ Flushing, NY 11367}
\email{nicholas.vlamis@qc.cuny.edu}


\begin{document}  

\begin{abstract}
Given a 2-manifold, a fundamental question to ask is which groups can be realized as the isometry group of a Riemannan metric of constant curvature on the manifold. In this paper, we give a nearly complete classification of such groups for infinite-genus 2-manifolds with no planar ends. Surprisingly, we show there is an uncountable class of such 2-manifolds where every countable group can be realized as an isometry group (namely, those with self-similar end spaces). We apply this result to obtain obstructions to standard group theoretic properties for the groups of homeomorphisms, diffeomorphisms, and the mapping class groups of such 2-manifolds.  For example, none of these groups satisfy the Tits Alternative; are coherent; are linear; are cyclically or linearly orderable; or are residually finite. As a second application, we give an algebraic rigidity result for mapping class groups. 
\end{abstract}

\maketitle

\section{Introduction}

Allcock in \cite{AllcockHyperbolic} and Winkelmann in \cite{Winkelmann} independently proved that every countable group is realized as the isometry group of some complete hyperbolic surface whose topology depends on the group. 
Prior to their work, it was known that every finite group can be realized as the isometry group of a closed hyperbolic manifold in every dimension: Greenberg \cite{Greenberg} established dimension two, Kojima \cite{Kojima} dimension three, and Belolipetsky--Lubotzky \cite{BL}---following the work of Long and Reid \cite{LR}---the remaining cases.
Allcock's proof is constructive: given a countable group $G$, he explicitly constructs the desired hyperbolic surface $X$, which has finite or infinite genus whenever \( G \) has finite or infinite cardinality, respectively.
Since the topology of \(X\) depends on \(G\), one can ask in the other direction for restrictions on groups that act by isometries on a complete hyperbolic surface of fixed topological type. 
A well-known theorem of Hurwitz gives an upper bound on the order of a group acting on a closed hyperbolic surface in terms of its genus.
In this paper we determine what, if any, restrictions there are on isometry groups of infinite-genus hyperbolic surfaces. More specifically, we propose: 

\begin{problem} \label{main question}
Let $S$ be a 2-manifold \footnote{The terms 2-manifold and surface are used throughout the paper. We use the term 2-manifold when we want to specify that the manifold is without boundary.}. Characterize the groups $G$ for which there exists a complete hyperbolic metric on $S$ whose isometry group is isomorphic to $G$.
\end{problem}

A complete characterization of such groups \(G\) for a closed surface remains open; the main focus of this paper is to address Problem~\ref{main question} in the infinite-genus setting. The terms \textit{self-similar, doubly pointed} and \textit{non-displaceable subsurface} in the theorem below will all be defined in Sections~\ref{Denumerable} and \ref{sec:SSandPS}, but, very roughly, they are meant to generalize the situation of an infinite-genus surface having one, two, or three ends, respectively (see Figure \ref{fig:exs}).

\begin{ThmA}
Let \( S \) be an orientable infinite-genus 2-manifold with no planar ends and let $G$ be an arbitrary group. Then: 
\begin{enumerate}
\item If the end space of $S$ is self-similar, there exists a  complete hyperbolic metric on $S$ whose isometry group is $G$ if and only if $G$ is countable. 

\item If the end space $S$ is doubly pointed, then the isometry group of any complete hyperbolic metric on $S$ is virtually cyclic. 

\item If $S$ contains a compact non-displaceable subsurface, then there exists a  complete hyperbolic metric on $S$ whose isometry group is $G$ if and only if $G$ is finite. 

\end{enumerate}
Moreover, every such 2-manifold satisfies at least one of the above hypotheses.
\end{ThmA}

We note that Theorem~\ref{mainUncountable} is not a trichotomy due to the fact that the second and third statements are not mutually exclusive. 
The proof of Theorems~\ref{mainUncountable} and \ref{thm:partialclass} rely on the results of Section~\ref{sec:SSandPS}. In particular, Theorem~\ref{mainUncountable}(1) is first proved for 2-manifolds with end spaces that have \emph{radial symmetry}, defined in Section~\ref{Denumerable}. In Section~\ref{sec:SSandPS}, we prove that radial symmetry of the end space of $S$ is equivalent to the notion of self-similarity (introduced in \cite{MannRafi}), which is a necessary component in proving Theorem~\ref{thm:partialclass} and Theorem~\ref{mainUncountable}. Both of the assumptions that \(S\) has no planar ends and has infinite genus are justified, in part, in Section~\ref{sec:other surfaces}. Note that there are aspects of the proof of Theorem~\ref{mainUncountable} for which one or both the assumptions are not necessary (this is indicated throughout the text). However, for the full statement of Theorem~\ref{mainUncountable}, we rely on the given restrictions.

We would like to stress that, to the authors' knowledge, Theorem~\ref{mainUncountable} gives the first examples of surfaces for which every countable group can be realized as the isometry group of a complete hyperbolic metric on the surface. Note that both Allock's and Winkelmann's constructions yield finite-type surfaces whenever the group in question is finite.

When $S$ has a countable end space, we can strengthen Theorem~\ref{mainUncountable}(2), and characterize \emph{all} possible isometry groups in terms of the topology of $S$. In the statement of the theorem, the characterisic system $(\alpha, n)$ of the end space \(\cE\) of \(S\) is given by its Cantor--Bendixson rank and degree, which are defined in Section~\ref{sec:CB}. 
The case \( n = 2 \) is equivalent to being doubly pointed. Therefore, Theorem~\ref{main} below demonstrates that when \(\cE\) is countable, we can separate 2-manifolds with doubly pointed end spaces into two categories, depending on whether $\alpha$ is a limit or successor ordinal, and completely determine their possible isometry groups.

\begin{ThmB}
Let \( S \) be an orientable infinite-genus 2-manifold with a countable space of ends, none of which are planar, and let \( G \) be an arbitrary group.
Suppose the end space \( \cE \) of \( S \) has characteristic system \( (\al, n) \).
\begin{enumerate}
\item 
If \( n = 1 \), there exists a  complete hyperbolic metric on \( S \) whose isometry group is \( G \) if and only if \( G \) is countable.
\item 
If \( n = 2 \) and $\alpha$ is a successor ordinal, there exists a  complete hyperbolic metric on \( S \) whose isometry group is \( G \) if and only if \( G \) is virtually cyclic.
\item
If \( n \geq 3 \), or $n=2$ and $\alpha$ is a limit ordinal, there exists a complete hyperbolic metric on \( S \) whose isometry group is \( G \) if and only if \( G \) is finite.
\end{enumerate}
\end{ThmB}

\begin{figure}[h]
\begin{center}
\begin{overpic}[trim = 1.25in 8in 1in 1.25in, clip=true, totalheight=0.2\textheight]{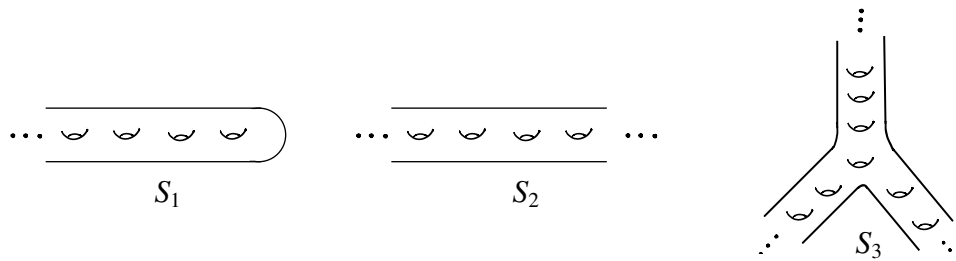}
\put(15,5){\large{$S_1$}}
\put(50,5){\large{$S_2$}}
\put(83.5,0){\large{$S_3$}}

\end{overpic}
\caption{Examples for the three cases of Theorems~\ref{mainUncountable} and \ref{main}.}\label{fig:exs}
\end{center}
\end{figure}

The 2-manifolds $S_1$, $S_2$, and $S_3$ in Figure \ref{fig:exs} are examples of cases (1), (2), and (3), respectively, of Theorems~\ref{mainUncountable} and \ref{main}. The class of 2-manifolds satisfying the criteria for each of the three cases in both of these theorems is uncountable. 

The main idea towards realizing groups as isometry groups in Theorems~\ref{mainUncountable} and \ref{main} is to build more rigidity into Allcock's construction so that, under the right circumstances, the topology of $S$ can be prescribed. In the other direction, we develop  obstructions for a group to act by isometries on a complete hyperbolic surface in terms of its underlying topology. 

The infinite-genus assumption for \(S\) in Theorems~\ref{mainUncountable} and \ref{main} is necessary: in Proposition~\ref{prop:finitegenus} we show that for every orientable finite-genus 2-manifold, there exists a finite group not realizable as the isometry group of a complete hyperbolic metric on the manifold. 
We remark in several places later in the article that the condition on (the lack of) planar ends is stronger than necessary; however, there needs to be some condition along these lines.
It is not clear what the optimal condition should be, but, in Proposition~\ref{prop:finiteplanar}, we show that if a 2-manifold has a finite number of planar ends, then, again, there exists a finite group not realizable as the isometry group of a complete hyperbolic metric on the manifold. 

\begin{Rem}
We pause to note two natural strengthenings of Theorem~\ref{mainUncountable}  that follow from our constructions.
Let \( S \) and \( G \) be as in Theorem \ref{mainUncountable} and let \( X = X(S,G) \) denote the complete hyperbolic surface given by Theorem \ref{mainUncountable}.

(1) \( X \) contains no funnels or half-planes---equivalently, \( X \) is equal to its convex core (see the remark following Lemma \ref{lem:geod-complete} and the discussion around Proposition \ref{prop:complete}).

(2) Every isometry of \( X \) is orientation preserving (see the note after Lemma \ref{lem:isom}).
Therefore, thinking of \( X \) as a Riemann surface, \( \Aut(X) \), the group of biholomorphisms \( X \to X \), is isomorphic to \( G \); in particular, Theorem \ref{mainUncountable} can be rephrased in terms of automorphisms of (hyperbolic) Riemann surfaces. 
\end{Rem}

Another consequence of the proofs of Theorems~\ref{mainUncountable} and \ref{main} is the following: 

\begin{ThmC}
Let \( S \) be an orientable infinite-genus 2-manifold with no planar ends.
If  \( G \) is any finite group, then there exists a regular covering \( \pi \co S \to S \) such that the deck group associated to \( \pi \) is isomorphic to \( G \). If, additionally, the end space of \(S\) is self-similar and \( G \) is any countable group, then there exists a regular covering \( \pi \co S \to S \) such that the deck group associated to \( \pi \) is isomorphic to \( G \). 
\end{ThmC}

Note that if \( S \) is an infinite-genus 2-manifold then its fundamental group is a free group with countably infinite rank.
Therefore, it follows that given any countable group \( G \), there exists a regular covering \( S' \to S \) whose deck group is isomorphic to \( G\); therefore, the content of Theorem~\ref{thm:covering} is to provide conditions under which we can guarantee that \( S' \) is homeomorphic to \( S \).

Before moving to applications, we note that Theorem~\ref{main}(1) has an analog in the setting of Veech groups associated to translations surfaces; in particular, if \( S \) and \( G \) are as in Theorem~\ref{main}(1), with the additional restriction that \(\alpha\) is finite, then Maluendas--Valdez \cite{Valdez2} show that there exists a translation structure on \( S \) with Veech group isomorphic to \( G \) if and only if \( G \) is isomorphic to a countable subgroup of \( \rm{GL}(2,\br) \) (they also show this for the case when the end space is a Cantor space).
This work generalizes earlier work of Przytycki--Schmith\"usen--Valdez \cite{Valdez1} where they show the same result in the single case of the Loch Ness monster surface.

In the process of proving Theorem~\ref{mainUncountable}, we prove that for any infinite-genus 2-manifold $S$ with no planar ends, any finite group $G$ is realizable as the isometry group for some complete hyperbolic metric on \( S \) (see Theorem~\ref{thm:finite}). We immediately have the following corollaries to Theorem~\ref{mainUncountable}, where $\Homeo(S)$, $\Diffeo(S),$ and $\mcg(S)$ are the homemorphism group, diffeomorphism group, and mapping class group of $S$, respectively.

\begin{Cor} \label{cor:finite}
If $S$ is an orientable infinite-genus 2-manifold with no planar ends, then $\Homeo(S), \Diffeo(S),$ and $\mcg(S)$ each contain an isomorphic copy of every finite group.
\end{Cor}

\begin{Cor} \label{completely realizable}
Let $S$ be an orientable infinite-genus 2-manifold with no planar ends. If the end space of $S$ is self-similar, then $\Homeo(S), \Diffeo(S),$ and $\mcg(S)$ each contain an isomorphic copy of every countable group.
\end{Cor}

Given Corollary~\ref{cor:finite}, it is natural to ask if for every finite subgroup $H < \mcg(S)$, there exists \( H' < \rm{Homeo}^+(S) \) such that the natural projection \( \rm{Homeo}^+(S) \to \mcg(S) \) restricts to an isomorphism \( H' \to H \)---this is known as the Nielsen realization problem. 
For finite-type surfaces, Nielsen realization was proved by Kerckhoff \cite{Kerckhoff}, and Afton--Calegari--Chen--Lyman \cite{ACCL} recently proved an analog of the Nielsen realization theorem for mapping class groups of infinite-type surfaces. 
In fact, \( H' \) can be taken to be a subgroup of the isometry group of a complete hyperbolic metric on \( S \). 
We also note that Markovic \cite{Markovic} proved a version for the Teichm\"uller modular group: given a hyperbolic Riemann surface \( X \) and a group acting on \( \rm{Teich}(X) \), the Teichm\"uller space of \( X \), with bounded orbits, the group must fix a point in \( \rm{Teich(X)} \).
If \( X \) is closed, then this follows from Kerckhoff's result; hence, Markovic's result is another version of Nielsen realization in the infininite-type setting.

Corollary~\ref{completely realizable} can be leveraged to determine whether or not $\Homeo(S)$, $\Diffeo(S)$, and $\mcg(S)$ have certain algebraic properties. Suppose that $P$ is a hereditary property of groups---that is, if a group $G$ has property $P$, then every subgroup of $G$ has property $P$--- then, if there exists a subgroup $H$ of $G$ that does not have property $P$, $G$ does not have property $P$. A first application of our work is: 

\begin{Cor}\label{cor:hereditary}
Let $S$ be an orientable infinite-genus 2-manifold with no planar ends and a self-similar end space. If $P$ is any hereditary property such that there exists a countable group $H$ that does not possess property $P$, then $\Diffeo(S), \Homeo(S),$ and $\mcg(S)$ do not have property P. In particular, $\Diffeo(S), \Homeo(S),$ and $\mcg(S)$:

\begin{itemize}
\item do not satisfy the Tits alternative.
\item are not linear.
\item are not residually finite.
\item are not (cyclically or linearly) orderable.
\item are not coherent.
\end{itemize}

\end{Cor}

 
 The fact that $\mcg(S)$ is not residually finite for any orientable infinite-genus surface $S$ was first proved by the second two authors \cite{PatelAlgebraic}.  To the best of our knowledge, in most cases the other four applications were otherwise previously unknown (some cases can be deduced from work of Maluendas--Valdez \cite{Valdez2} cited above). We note that, in upcoming work, Allcock proves that $\mcg(S)$ does not satisfy the Tits Alternative for \emph{any} infinite-type 2-manifold $S$. Additionally, Lanier--Loving show in \cite{LanierLoving} that mapping class groups of orientable infinite-type 2-manifolds do not satisfy the strong Tits Alternative, which does not have any implication for the (classical) Tits alternative. 

The study of mapping class groups of infinite-type surfaces (so-called \textit{big mapping class groups}) is growing rapidly, and the first application given by Corollary~\ref{cor:hereditary} fits into a large body of work studying the algebraic properties of big mapping class groups. We refer the reader to the recent survey of Aramayona and the third author for a more detailed summary of related work \cite{AramayonaVlamis}.

The second main application of our work is an algebraic rigidity result for big mapping class groups. The pure mapping class group, denoted $\pmcg(S)$, is the kernel of the action of $\mcg(S)$ on the space of ends of \( S \). For finite-type surfaces, it is well-known that there are a few algebraic invariants of $\mcg(S)$ or $\pmcg(S)$ that determine the topology of $S$ (see Section \ref{sec:app2} for details). Using several recent results  and Theorem~\ref{thm:finite}, we give a finite list of algebraic invariants of $\pmcg(S)$ that completely determine the topology of $S$ for a countable collection of infinite-type surfaces:

\begin{ThmD}
Let \( \Omega_n \) be an \( n \)-ended orientable infinite-genus 2-manifold with no planar ends and let \( G = \pmcg(S) \) for some orientable  2-manifold \( S \).
Then, \( S \) is homeomorphic to \( \Omega_n \) if and only if  \( G  \) satisfies each of the following conditions:
\begin{enumerate}[(i)]
\item \( G \) is not residually finite,
\item \( G \) is not cyclically orderable,
\item \( \mathrm{Hom}(G, \bz) \) has algebraic rank \( n-1 \), and
\item \( G \) is finite index in \( \mathrm{Aut}(G) \).
\end{enumerate}
\end{ThmD}

\subsection*{Outline} 
In Section~\ref{Prelim}, we cover the relevant preliminaries on hyperbolic geometry, the topology of infinite-type surfaces, spaces of ends, and Cantor--Bendixson rank and degree. In Section~\ref{Allcock}, we build on Allcock's argument to first prove that for any finite group $G$ and any orientable infinite-genus 2-manifold $S$ with no planar ends, $G$ is realizable as the isometry group for some complete hyperbolic metric on $S$. We then prove that the unique (up to homeomorphism) orientable 2-manifold $L$ with infinite genus and one end, often referred to as the \textit{Loch Ness monster surface}, has the property that every countable group is realized as an isometry group of a complete hyperbolic metric on \( L \). In Section~\ref{Denumerable}, we prove the remaining components of Theorem~\ref{main} and several of the components of Theorem~\ref{mainUncountable}. We prove the equivalence of radial symmetry and self-symmetry of the end space of any orientable 2-manifold in Section~\ref{sec:SSandPS}. We use this equivalence to  prove Theorem~\ref{thm:partialclass} and finish the proof of Theorem~\ref{mainUncountable} in Section~\ref{sec:trichotomy}. In Section~\ref{sec:covers}, we finish the proof of Theorem~\ref{thm:covering}. We prove the algebraic rigidity application for big mapping class groups in Section~\ref{sec:app2}, and finally, in Section~\ref{sec:other surfaces} we justify the assumptions on \(S\) in Theorems~\ref{mainUncountable} and \ref{main} with Propositions~\ref{prop:finitegenus} and \ref{prop:finiteplanar}. 

\subsection*{Acknowledgements}
\vspace{-12pt}
The authors thank the referee for carefully reading the article and for their helpful suggestions, which have no doubt benefited the exposition.
The authors would also like to thank Kathryn Mann for her suggestions on the proofs contained in Section~\ref{sec:SSandPS} and numerous helpful discussions. We would also like Kathryn Mann and Kasra Rafi for sharing an early draft of their paper \cite{MannRafi} that was very helpful for writing Section~\ref{sec:SSandPS}. The first author was supported by NSF DMS-1807319, the second author was supported by NSF DMS 1812014 \& 1937969, and third author was supported in part by NSF RTG 1045119 and PSC-CUNY grant 62571-00~50.


\section{Preliminaries} \label{Prelim}
A 2-manifold is a connected, second countable, Hausdorff topological space in which every point has a neighborhood homeomorphic to \( \br^2 \). In particular, a 2-manifold is without boundary unless otherwise specified.
A Riemannian metric on a 2-manifold is \emph{hyperbolic} if it has constant sectional curvature \( -1 \);
a \emph{hyperbolic surface} is a connected 2-manifold equipped with a hyperbolic metric; a hyperbolic surface is \emph{complete} if it is a complete metric space.
Up to isometry, there is a unique complete simply connected hyperbolic surface, called the \emph{hyperbolic plane} and denoted \( \bh \). 
It is convenient to have a concrete model: we will identity \( \bh \) with the upper-half plane \( \{z = x+iy \in \bc :  y>0 \} \) equipped with the metric \( ds^2 = \frac{dzd\bar z}{y^2} \). 
A discrete group of isometries of \( \bh \) is called a \emph{Fuchsian} group.
Every complete hyperbolic surface is a quotient of the hyperbolic plane by an action of a Fuchsian group.
It follows that an equivalent definition of a hyperbolic surface is a connected metric space in which every point has an open neighborhood  isometric to an open subset of \( \bh \).

In the constructions below, we will need to extend our definition above to allow for hyperbolic surfaces with boundary:
A \emph{hyperbolic surface with totally geodesic boundary} is a connected complete metric space in which every point has a neighborhood  isometric to an open subset of a closed (geodesic) half-plane in \( \bh \).

In the following subsections, we will recall the definition of ends of a topological space, the classification of (topological) surfaces,  the classification of countable compact Hausdorff topological spaces, and some standard facts about hyperbolic geometry.

\subsection{Topological ends} 

The notion of a topological end captures  topologically distinct ways to escape a space.
For us, and historically, the introduction of topological ends serves two key purposes: classifying topological surfaces and building invariants of groups. 

For the remainder of the subsection, let \( X \) denote a topological space that is connected, locally connected, locally compact, second countable, and Hausdorff. 
An \emph{exiting sequence} in \( X \) is a sequence \( \{U_n\}_{n\in\bn} \) of open subsets of \( X \) satisfying:
\begin{enumerate}
\item \( \partial U_n \) is compact for each \( n \in \bn \),
\item \( U_{n+1} \subset U_n \) for each \( n \in \bn \), and
\item for any compact subset \( K \) of \( X \) there exists \( N \in \bn \) such that \( U_N \cap K = \varnothing \).
\end{enumerate}
Two exiting sequences \( \{U_n\}_{n\in\bn} \) and \( \{V_n\}_{n\in\bn} \) are equivalent if for every \( n \in \bn \) there exists \( m \in \bn \) such that \( U_m \subset V_n \) and \( V_m \subset U_n \).

An \emph{end} of \( X \) is an equivalence class of exiting sequences.
The \textit{space of ends}, denoted \( \cE (X) \) (or simply \( \cE \) when the underlying space is clear), is, as a set, the set of all ends of \( X \).
Let \( V \) be an open subset of \( X \) with compact boundary, define \( \widehat V = \left\{ \left[\{U_n\}_{n\in\bn}\right]\in \cE : U_n \subset V \text{ for some } n\in\bn\right\} \) and let \( \mathcal V = \{ \widehat V : V \subset X \text{ is open with compact boundary}\} \).
The set \( \cE \) becomes a topological space by declaring \( \mathcal V \) a basis. 
With this topology, the space of ends is always compact, second countable, totally disconnected, and Hausdorff (see \cite[Sections 36--37, Chapter 1]{AhlforsSario} or \cite[Section 2.1]{FernandezEnds}); in particular, it is homeomorphic to a closed subset of the Cantor set (see \cite[Theorem 7.8]{KechrisClassical}).

On several occasions, we will use the following proposition for inducing end maps, which can readily be argued from definitions:

\begin{Prop}
\label{prop:end-map}
Let \( X \) and \( X' \) be connected, locally connected, locally compact, second countable,  Hausdorff topological spaces.
If  \( f \co X \to X' \) is a proper continuous map, then there exists a unique continuous function \( \widehat f \co \cE(X) \to \cE(X') \) satisfying \( \widehat f^{-1}(\widehat W) = \widehat{f^{-1}(W)} \) for every open subset \( W \) of \( X' \) with compact boundary. \qed
\end{Prop}

As a corollary, we see that the function \( \Homeo(X) \to \Homeo(\cE(X)) \) given by \( f \mapsto \widehat f \) is a homomorphism.
(The existence of this homomorphism is also clear from definitions.)

We will also use the following well-known proposition, which is used to define the end of a topological group. 

\begin{Prop}
\label{prop:end-group}
Let \( G \) be a Hausdorff topological group and let \( X \) and \( X ' \) be  connected, locally connected, locally compact, second countable,  Hausdorff topological spaces.
If \( G \) acts properly discontinuously and cocompactly on both \( X \) and \( X' \) by homeomorphisms, then \( \cE(X) \) and \( \cE(X') \) are homeomorphic. \qed
\end{Prop}

\subsection{The classification of surfaces}

We will need to consider bordered manifolds: we define a \emph{(topological) surface}\footnote{A surface is permitted to have empty boundary. When we want to exclude the case of nonempty boundary, we will use the term 2-manifold.} to be a connected, second countable, Hausdorff topological space in which every point has a neighborhood homeomorphic to an open subset of \( \{(x,y) \in \br^2 : y \geq 0\} \subset \br^2\).
A (topological) surface is of \emph{finite (topological) type} if its fundamental group is finitely generated; otherwise, it is of \emph{infinite (topological) type}.

It is well known that closed orientable surfaces are classified by their genus.
Surprisingly, there is a complete classification of all surfaces. 
The essential tool in this classification is the notion of a topological end introduced above.

An end \( \left[\{U_n\}_{n\in\bn}\right] \)  is \textit{planar} if there exists \( n \in \bn \) such $U_{n}$ is homeomorphic to an open subset of the plane \( \br^2 \) (or, equivalently, there exists \( n \in \bn \) such that \( U_n \) has finite genus); otherwise, it is \emph{non-planar}.
For an orientable surface, an end \( \left[\{U_n\}_{n\in\bn}\right] \) is  non-planar if and only if each \( U_n \) has infinite genus. The set of non-planar ends is denoted by \(\cE_{np}\) and is a closed subset of \(\cE\). An end is \textit{isolated} if it is an isolated point of the space of ends.

Building off the classification of compact surfaces, Ker\'{e}kj\'{a}rt\'{o} gave the first proof of the classification of 2-manifolds in general---the standard reference for a proof is \cite{RichardsClassification}.
For our purposes, we state the classification for orientable surfaces.

\begin{Thm}[Classification of orientable 2-manifolds]
Let \( S \) and \( S' \) be orientable  2-manifolds of the same---possibly infinite---genus.
Then, \( S \) and \( S' \) are homeomorphic if and only if there is a homeomorphism \( \cE(S) \to \cE(S') \) sending \( \cE_{np}(S) \) onto \( \cE_{np}(S') \). 
\end{Thm}

We will need to use a slightly more general version of the classification allowing for surfaces with compact boundary components.
In this case, let \( \cE_\partial(S) \) denote the subset of \( \cE(S) \) containing all ends \( \left[\{U_n\}_{n\in\bn}\right] \) such that, for all \( n \), \( U_n \) contains a component of \( \partial S \). 
It is not difficult to deduce from Ker\'ekj\'art\'o's theorem that two orientable surfaces with compact boundary components, \( S \) and \( S' \), of the same (possibly infinite) genus and with the same (possibly infinite) number of boundary components are homeomorphic if and only if there is a homeomorphism  \( \cE(S) \to \cE(S') \) mapping \( \cE_{np}(S) \) onto \( \cE_{np}(S') \) and \( \cE_\partial (S) \) onto \( \cE_\partial(S') \) (see \cite{PrishlyakClassification} for a complete classification of surfaces).


\subsection{Countable compact Hausdorff spaces}\label{sec:CB}

In this section we recall the classification of countable compact Hausdorff spaces, which we use to give a strengthening of our main theorem (see Theorem~\ref{main}) for surfaces with countable end spaces. 

We identify an ordinal \( \alpha \) with the set of all ordinals strictly less than \( \al \).
Following standard notation, we let \( \omega \) denote the first countably infinite ordinal (so, \( \omega = \bn \cup\{0\} \)\footnote{For clarity of notation, we specify \(0\notin\bn\).}) and \( \omega_1 \) the first uncountable ordinal.
In this perspective, an ordinal \( \alpha \) can be viewed as a topological space by equipping the set \( \al \) with the order topology.
Given a topological space \( X \) and a subset \( Y \) of \( X \), let \( Y' \) be the set consisting of all the accumulation points of \( Y \) in \( X \).
Given an ordinal \( \al \), the \emph{\( \al^{\rm{th}} \) Cantor--Bendixson derivative of \( X \)}, denoted \( X^{\al} \), is defined inductively as follows:

\begin{itemize}
\item \( X^0 = X \),
\item \( X^{\al+1}  = (X^{\al})' \), and 
\item \( X^\lambda = \bigcap_{\be < \lambda} X^\be \) for a limit ordinal \( \lambda \).
\end{itemize}

The \emph{Cantor--Bendixson rank} of a Hausdorff space \( X \) is the smallest ordinal \( \beta \) satisfying \( X^\beta = X^{\beta+1} \) (such an ordinal always exists).
If \( X \) is countable with Cantor--Bendixson rank \( \beta \), then \( \beta \) is a countable ordinal and \( X^\beta = \varnothing \). 
Moreover, if \( X \) is compact, then the Cantor--Bendixson rank is necessarily a successor ordinal, i.e. \(\beta = \al +1 \) for some ordinal \(\alpha\).
Therefore, if \( X \) is countable, compact, and has Cantor--Bendixson rank \( \al+1 \), then \( X^\al \) is finite; the cardinality of \( X^\al \) is called the \emph{Cantor--Bendixson degree} of \( X \). 
By a theorem of Mazurkiewicz and Sierpinski \cite{MazurkiewiczContribution}, a countable, compact, Hausdorff space is determined up to homeomorphism by its Cantor--Bendixson rank and degree:

\begin{Thm}
\label{thm:classification}
Let \( A \) be a countable, compact, Hausdorff topological space.
If \( A \) has Cantor--Bendixson rank \( \al+1 \in \omega_1 \) and Cantor--Bendixson degree \( d \in \omega \),  then \( A \) is homeomorphic to \( \omega^\al \cdot d + 1 \).  
\qed
\end{Thm}

\subsection{Hyperbolic geometry}

In this section, we gather standard facts about the geometry of hyperbolic surfaces and their isometry groups. For references on hyperbolic geometry see \cite{HubbardTeichmuller1} and \cite{Buser}.
A quick comment about notation: given a hyperbolic  surface \( X \), we use \( d \) to denote its metric, \( \ell(\gamma) \) to denote the length of a geodesic \( \gamma \) in \( X \), and \( \isom(X) \) to denote the group of isometries of \( X \).

The goal of the first three lemmas is to establish basic facts about closed geodesics, the action of the isometry group on a hyperbolic surface, and the cardinality of isometry groups of hyperbolic surfaces.

\begin{Lem}
\label{lem:finite-scg}
On a hyperbolic  surface, the number of  closed geodesics of bounded length intersecting a compact set is finite.
\end{Lem}

\begin{proof}
Let \( \Gamma \) be a Fuchsian group and let \( X = \Gamma \backslash \bh \) be a hyperbolic surface.
Now let \( K \subset X \) be compact and \( b \) be a positive real number.
Suppose there exists a sequence of simple closed geodesics \( \{\gamma_n\}_{n\in\bn} \) in \( X \) such that \( \gamma_n\neq\gamma_m \), \( \ell(\gamma_n) < b \), and \( \gamma_n \cap K  \neq \varnothing \) for all distinct \( n,m \in \bn \).

Choose a lift \( \tilde K \) of \( K \) in \( \bh \) and lifts \( \tilde \gamma_n \) of \( \gamma_n \) such that \( \tilde\gamma_n \cap \tilde K \neq \varnothing \). 
Now let \( g_n \in \Gamma \) be a generator of the (setwise) stabilizer, in \( \Gamma \), of \( \tilde \gamma_n \).
If \( B \) is the closed \( b \)-neighborhood of \( \tilde K \), then \( B \) is compact and \( g_nB \cap B \neq \varnothing \) for all \( n \in \bn \), but, as every Fuchsian group acts on \( \bh \) properly discontinuously, this is a contradiction.
\end{proof}

\begin{Lem}
\label{lem:prop-discont}
The action of the isometry group of a complete hyperbolic surface with non-abelian fundamental group on the surface is properly discontinuous.
\end{Lem}

\begin{proof}
Let \( X \) be a complete hyperbolic surface such that \( \pi_1(X) \) is  non-abelian.
Given a compact subset \( K \) of \( X \), let \( I_K = \{f\in \isom(X) : f(K) \cap K \neq \varnothing\} \).
It is enough to show that \( I_K \) is finite.
By possibly enlarging \( K \), we may assume that there are two intersecting closed geodesics \( \gamma_1 \) and \( \gamma_2 \) contained in the interior of \( K \) (this is possible since \( \pi_1(X) \) is non-abelian).
Observe that \( \bigcup_{f\in I_K} f(K) \) has bounded diameter and hence is contained in some compact set \( K' \). 
It follows that all the \( I_K \)-translates of \( \gamma_1 \) and \( \gamma_2 \) are contained in \( K' \).
Orient \( \gamma_1 \) and let \( v \) be the unit tangent vector based at a point \( x \in \gamma_1 \cap \gamma_2 \) in the direction of \( \gamma_1 \).
The finiteness of \( \gamma_1 \cap \gamma_2 \), together with Lemma \ref{lem:finite-scg}, guarantees that the orbit of \( v \) under \( I_K \) is finite.  
But, if two isometries agree on a unit tangent vector, then they agree everywhere; hence,  \( I_K \) must be finite as well. 
\end{proof}

(N.B.  Lemma \ref{lem:prop-discont} is true for hyperbolic manifolds in general, for instance see \cite[Proposition V.E.10]{MaskitKleinian}.)

\begin{Lem}
\label{lem:countable}
The isometry group of a complete hyperbolic surface with non-abelian fundamental group is countable.
Moreover, the isometry group of every hyperbolic surface of finite topological type is finite.
\end{Lem}

\begin{proof}
Let \( X \) be a complete hyperbolic surface with non-abelian fundamental group. 
Recall that \( X \), being a metrizable manifold, is second countable and locally compact; in particular, \( X \) is a countable union of compact sets. 
Therefore, since, by Lemma \ref{lem:prop-discont}, \( \isom(X) \) acts on \( X \) properly discontinuously, 
\( \isom(X) \) is a countable union of finite sets and hence countable.

To finish, let \( X \) be of finite topological type.
Since \( X \) has finite genus, every end of \( X \) is planar.
In particular, as \( X \) has a finite number of ends, all of which are planar, we can find a compact subset \( K \) of \( X \) such that \( X \ssm K \) is a disjoint union of annuli.  
It follows that every closed geodesic in \( X \) must intersect \( K \); hence, \( f(K) \cap K \neq \varnothing \) for every \( f \in \isom(X) \).
Again, by Lemma \ref{lem:prop-discont}, \( \isom(X) \) is finite.
\end{proof}

A fundamental tool in 2-dimensional hyperbolic geometry is the \emph{collar lemma}---a very concise and elegant proof can be found in \cite[Theorem 3.8.3]{HubbardTeichmuller1}.
Due to its repeated use below, we recall its statement here in detail.

Let \( \gamma \) be a simple closed geodesic in a hyperbolic surface \( X \).  
If the \( \delta \)-neighborhood \( A_\delta(\gamma) = \{ x \in X : d(x,\gamma)<\delta\} \) of \( \gamma \) in \( X \) is isometric to the \( \delta \)-neighborhood of the unique simple closed geodesic in the annulus obtained as the quotient of  \( \bh \) by the Fuchsian group generated by the isometry \( z \mapsto e^{\ell(\gamma)}z \), then we say that \( \gamma \) admits a \( \delta \)-collar. 
Define the \emph{collar function} \( \eta \) by 
\[
\eta(\ell) = \arcsinh\left(\frac1{\sinh(\ell/2)}\right).
\]

\begin{Thm}[The collar lemma]
\label{thm:collar}
Let \( X \) be a complete hyperbolic surface.
If \( \{\gamma_i\}_{i\in I} \) is a collection of pairwise-disjoint simple closed geodesics indexed by a countable set \( I \), then the sets \( A_{\eta(\ell_i)}(\gamma_i) \) are pairwise-disjoint collars, where \( \ell_i = \ell(\gamma_i) \). \qed
\end{Thm}

In the sequel, we will  routinely appeal to the following corollary of the collar lemma (see \cite[Corollary 3.8.7]{HubbardTeichmuller1}).

\begin{Cor}
\label{cor:sinh}
On a hyperbolic surface, any two simple closed geodesics of length less than \( 2\arcsinh(1) \) are either equal or disjoint.
\qed
\end{Cor}

For constructing hyperbolic surfaces, it is useful to have a brief discussion of Fenchel--Nielsen coordinates.  
A \emph{pair of pants} is a 2-manifold homeomorphic to the 2-sphere with three points removed.
Given a surface $S$, a \textit{(topological) pants decomposition} is a maximal collection of homotopically non-trivial simple closed curves on $S$ that are pairwise disjoint, pairwise non-homotopic, and such that the complement of their union is a  disjoint union of pairs of pants.

Given a pants decomposition $\mathcal{P}$ on $S$, a hyperbolic metric on $S$ is determined by specifying a length $l_{\gamma}$ for each curve $\gamma$ in $\mathcal{P}$, as well as a \textit{twist parameter} $t_{\gamma} \in \mathbb{R}$ that keeps track of how the pairs of pants are glued together to obtain $S$. Intuitively, $t_{\gamma}$ specifies the amount by which the left hand side of $\gamma$ should be twisted before gluing to the right hand side; we refer the reader to \cite[Section 7.6]{HubbardTeichmuller1} or \cite[Section 1.7]{Buser} for more details and the formal definition.

Note that on a hyperbolic surface, a maximal collection of pairwise-disjoint simple closed geodesics need not be a pants decomposition---for instance, it is possible to have a funnel or a half-plane in the complement of such a collection (in fact, these are the only possibilities).
In such a case, the hyperbolic surface fails to be equal to its convex core, where
the \emph{convex core} of a hyperbolic surface \( X \) is the smallest closed convex subsurface (with boundary) containing every simple closed geodesic in \( X \).
We end with a proposition, which can be derived from the collar lemma and will be used in our construction:

\begin{Prop}
\label{prop:complete}
Let \( X \) be an orientable hyperbolic surface with no planar ends and let \( \mathcal P \) be a topological pants decomposition of \( X \).
If there exists \( M > 0 \) such that \( \ell_X(\gamma) < M \) for each \( \gamma \in \mathcal P \), then \( X \) is complete.
Moreover,  \( X = C(X) \). \qed
\end{Prop}

For a further discussion on completeness, geodesic completeness, and convex cores of (bordered) hyperbolic surfaces, we refer the interested reader to \cite{BasmajianSaric}. 


\section{Finite groups} \label{Allcock}

Let \( G \) be a nontrivial countable group and let \( S \) be an orientable infinite-genus 2-manifold with no planar ends.
In this section, we give an explicit construction---using a Cayley graph for \( G \) as a blueprint---of a complete hyperbolic surface \( X_S^G \) whose isometry group is isomorphic to \( G \).

In the case that \( G \) is finite, we will see that \( X_S^G \) is homeomorphic to \( S \), showing that every finite group is realized as the isometry group of some complete hyperbolic metric on \( S \) (Theorem~\ref{thm:finite}). We remark that it is well known that there always exists a hyperbolic metric on \(S\) with trivial isometry group (in fact we outline an argument for this in our construction below). 

This result for finite groups is the main focus of the section; however, in our construction, considering countable groups is no more difficult than restricting to finite groups. We will, therefore, work in the more general setting of countable groups.
The advantage of doing so is realized at the end of the section where we observe that, when \( G \) is infinite, \( X_S^G \) is always homeomorphic to the Loch Ness monster surface---the orientable one-ended infinite-genus 2-manifold.
As a consequence, we see that every countable group is realized as the isometry group of a hyperbolic metric on the Loch Ness monster surface (Theorem~\ref{thm:lochness}).

With \( S \) and \( G \) as given above, we turn to the construction.
For clarity, we include figures demonstrating the construction for a surface $S$ with $|\cE|=3$ and $G$ a finite group with \( |G| = 4 \).
Let \( C_G \) be the complete, directed, and labelled Cayley graph of \( G \), that is, \( C_G \) is the graph whose vertices correspond to the elements of \( G \) and for each ordered pair \( (g,h) \in \{ (g,h) \in G^2 : h\neq id\} \) there is a directed edge from \( g \) to \( gh \) labelled `\( h \)'. The group \( G \) acts on \( C_G \) by left multiplication; it is an exercise to see that this action yields an isomorphism between \( G \) and the group of automorphisms of \( C_G \) preserving the direction and labels of edges (Allcock gives a short proof in the first paragraph of the proof of the main theorem in \cite{AllcockHyperbolic}). We note that the group of automorphisms of the unlabelled, undirected Cayley graph could be larger than $G$. 

We will use \( C_G \) as gluing instructions for building the desired hyperbolic structure on \( S \).
This is the same idea used in Allcock's argument \cite{AllcockHyperbolic}; however, the extra difficulty comes from controlling the topology of the resulting surface.
To proceed, we need a collection of hyperbolic surfaces corresponding to each edge and vertex of \( C_G \).

We first construct the vertex surface: 
every infinite-genus surface \(S\) has a nonempty space of ends denoted by \( \cE \).
As \( \cE \) is homeomorphic to a closed subset of the Cantor set, we can identify \( \cE \) with a closed subset of the 2-sphere \( S^2 \); define the surface \( S_\cE^2 \) to be the complement of \( \cE \) in \( S^2 \).

Let \( \delta \in \omega+1 \) such that \( H_1(S_\cE^2, \bq) \) has dimension \( |\delta| \).
Standard algebraic topology guarantees the existence of a basis \( \{v_i\}_{i\in \delta} \) of \( H_1(S_\cE^2, \bq) \) such that \( v_i \) can be represented on the surface by a simple closed curve \( c_i \subset S_\cE^2 \) for each \( i \in \delta \).
In addition, the \( c_i \) can be chosen to be pairwise disjoint.

For each \( i \in \delta \), choose a regular neighborhood \( \nu_i \) of \( c_i \) such that \( \nu_i \cap \nu_k = \varnothing \) for all distinct \( i, k \in \delta \). 
Let \( S' \) denote the complement of \( \bigcup_{i\in \delta} \nu_i \) in \( S_\cE^2 \).
Choose an indexing set \( J \) of the connected components of \( S' \) so that we may write \( S' = \bigsqcup_{j\in J} S_j \).
By construction, every simple closed curve \( c \) in \( S_j \) is separating and one component of \( S_j \ssm c \) is compact.
Indeed, \( c \) is separating as \( S_j \) is planar and both components of \( S_j \ssm c \) failing to be compact would contradict \( \{v_i\} \) being a basis.
It follows that \( S_j \) is one-ended; furthermore, as \( S_j \) is planar, its topology is determined---up to homeomorphism---by the cardinality of the connected components of \( \partial S_j \).
As \( S_j \) is a closed subset of \( S \), the inclusion \( S_j \hookrightarrow S \) is proper and hence, by Proposition \ref{prop:end-map}, induces a continuous map \( \cE(S_j) \to \cE \).
We can therefore identify the unique end of \( S_j \) with a point in \( \cE \); observe that the union of such points is dense in \( \cE \).

\begin{figure}[h]
\begin{center}
\begin{overpic}[trim = .5in 3.2in 1in 1.25in, clip=true, totalheight=0.4\textheight]{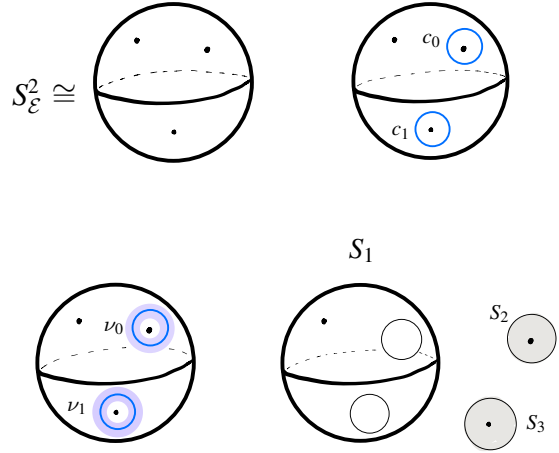}
\put(6,70){\large{$S^2_{\cE} \cong $}}
\put(68,65){\scriptsize{$c_1$}}
\put(73,80){\scriptsize{$c_0$}}
\put(15,20){\scriptsize{$\nu_1$}}
\put(21,33){\scriptsize{$\nu_0$}}
\put(61,45){\normalsize{$S_1$}}
\put(84,35){\scriptsize{$S_2$}}
\put(90,17){\scriptsize{$S_3$}}
\end{overpic}
\caption{For $S$ with $|\cE|=3$, we have the curves $c_0, c_1$, the neighborhoods $\nu_0, \nu_1$, and the surfaces $S_1, S_2,$ and $S_3$.}\label{fig:ex1}
\end{center}
\end{figure}

For \( \al \in \omega+1 \), define the surface \( Z_\alpha \) by 
\[
Z_\alpha = \br^2 \ssm \left( \bigcup_{m \in \alpha} B\left((m,0), \frac14\right) \right),
\]
where \( B(x,r) \) is the open ball of radius \( r \) about \( x \) in \( \br^2 \).
For \( j \in J \), let \( \alpha_j \in \omega+1 \) such that the set of boundary components of \( \partial S_j \) has cardinality \( |\alpha_j| \).
The classification of surfaces implies that \( S_j \) is homeomorphic to \( Z_{\alpha_j} \). Now, for each \( j \in J \), fix a homeomorphism \( \vp_j \co Z_{\alpha_j} \to S_j \) (see Figure~\ref{fig:ex2}).

\begin{figure}[h]
\begin{center}
\begin{overpic}[trim = .5in 7.2in .5in 1.5in, clip=true, totalheight=0.2\textheight]{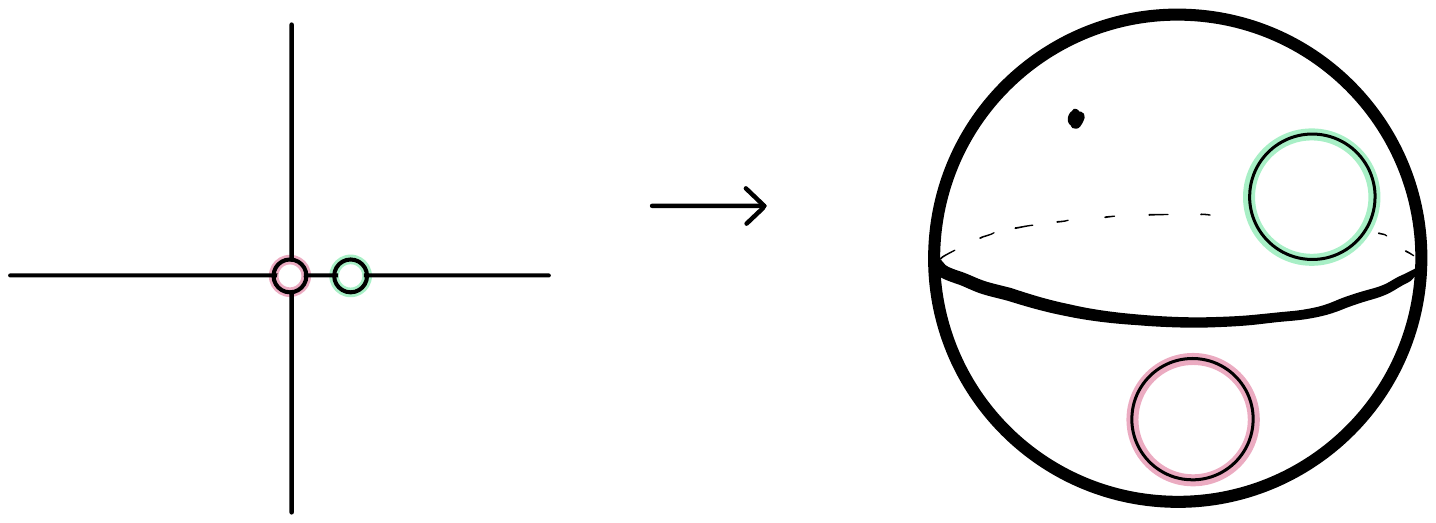}
\put(44,20){\normalsize{$\vp_1$}}
\put(10,20){\normalsize{$Z_2$}}
\put(85,17){\normalsize{$S_1$}}
\end{overpic}
\caption{The map $\vp_1$, where we imagine the point at $\infty$ of $Z_2$ is sent to the unique end of $S_1$.}\label{fig:ex2}
\end{center}
\end{figure}

As \( G \) is countable, we may choose an injective function (on the level of sets) \( f \co G \to \bz \). 
For \( \al \in \omega+1 \), define the surface \( Z_\alpha^G \) by
\[
Z_\alpha^G = Z_\alpha \ssm \left( \bigcup_{h\in G} \left( \bigcup_{m \in \bn} B\left( (f(h), m), \frac14\right) \right) \right).
\]

In other words, \(Z_\alpha^G\) is obtained from \(Z_\alpha\) by deleting a vertical column of disks for each element of \(G\). Again viewing \( S_j \) as a subsurface of \( S_\cE^2 \) and \( Z_\alpha^G \) as a subsurface of \( Z_\alpha \), we define the surface \( R_\cE \) to be 
\[
R_\cE = \left( \bigcup_{j\in J} \vp_j \left(Z_{\alpha_j}^G\right) \right) \cup \left( \bigcup_{i\in \delta} \nu_i \right).
\]
The components of \( \partial R_\cE \) are indexed by the set \( J \times G \times \bn \), namely, the component \( \partial(j,h,m) \) of \( \partial R_\cE \) corresponds to the simple closed curve contained in \( S_j \) whose center in \( \vp^{-1}(S_j) \) has coordinates \( (f(h), m) \). See Figure~\ref{fig:ex5}.

\begin{figure}[h]
\begin{center}
\begin{overpic}[trim = .5in 6.5in .5in 1.5in, clip=true, totalheight=0.27\textheight]{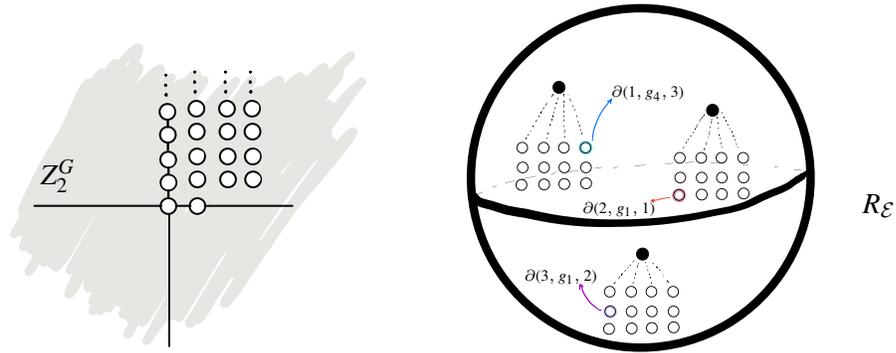}

\put(10,20){\normalsize{$Z^G_2$}}
\put(95,17){\normalsize{$R_\cE$}}
\put(69,29){\tiny{$\partial(1, g_4, 3)$}}
\put(66,17){\tiny{$\partial(2, g_1, 1)$}}
\put(60,10){\tiny{$\partial(3, g_1, 2)$}}
\end{overpic}
\caption{On the left hand side is $Z^G_2$, where \(|G| = 4\), and on the right hand side is $R_\cE$.}\label{fig:ex5}
\end{center}
\end{figure}

\begin{Lem}
\label{lem:end-space}
\( \cE(R_\cE) \) is homeomorphic to \( \cE \) and \( \cE_{\partial}(R_\cE) = \cE(R_\cE) \).
\end{Lem}

\begin{proof}
Note that we have constructed \( R_\cE \) as a subset of the 2-sphere and observe that, by construction, the closure of \( R_\cE \) in \( S^2 \) is \( R_\cE \cup \cE \).
It follows that \( R_\cE \) is a closed subset of \( S^2_\cE \); hence, the inclusion \( \iota \co R_\cE \hookrightarrow S^2_\cE \) is proper.
By Proposition~\ref{prop:end-map}, \( \iota \) induces a continuous map \( \widehat \iota \) from the end space of \( R_\cE \) to the end space of \( S^2_\cE \), which we identify with \( \cE \). 
Observe that, by construction, every neighborhood of a point in \( \cE \) contains a component of \( \partial R_\cE \); in particular, \( \widehat \iota(\cE_\partial(R_\cE)) = \widehat \iota(\cE(R_\cE)) \). 
It also follows that \( \widehat \iota \) is surjective.
Observe that every separating set in \( R_\cE \) is also separating in \( S^2_\cE \) and, therefore, \( \widehat \iota \) is injective.
We have established \( \widehat \iota \) is a continuous bijective map between compact Hausdorff spaces and hence a homeomorphism. 
This also allows us to conclude that \( \cE_{\partial}(R_\cE) = \cE(R_\cE) \).
\end{proof}

We need to build a hyperbolic metric on \( R_\cE \). To do so, we will use Fenchel--Nielsen coordinates.
Choose an injective function \[ \lambda \co J\times G \times \bn \to (0, \arcsinh(1)) \subset \br. \]
Fix a pants decomposition \( \mathcal Q \) on \( R_\cE \) (note that every component of \( \partial R_\cE \) is necessarily in \( \mathcal Q \)).
Let \( V \) be any totally geodesic hyperbolic surface homeomorphic to \( R_\cE \) such that (1) the length of \( \partial(j,h,m) \) is \( \lambda(j,h,m) \), (2) the length of non-peripheral curves in \( \mathcal Q \) have pairwise-distinct lengths in the interval \( (\arcsinh(1), 2\arcsinh(1)) \), and (3) such that no isometry of a pair of pants in the complement of \( \mathcal Q \) extends globally to \( V \) (this is readily accomplished by deforming the twist parameters).

We claim that \( V \) admits no isometries: observe that, by the collar lemma, every closed geodesic on \( V \) and not in \( \mathcal Q \) must have length at least \( 2\arcsinh(1) \); in particular, any isometry must setwise fix each  geodesic in \( \mathcal Q \) and hence induces an isometry on each individual pair of pants in the complement of \( \mathcal Q \), which, by assumption, cannot be extended to \( V \).
Therefore, \( V \) has no isometries.

Now let us turn to the edge surfaces.
Let \( E(j,h,2m) \) be a hyperbolic surface homeomorphic to a torus with two boundary components obtained by gluing together two pairs of pants, one with boundary lengths \( \lambda(j,h,2m), \arcsinh(1), \) and \( \arcsinh(1) \), and the other with lengths \( \lambda(j,h,2m-1), \arcsinh(1), \) and \( \arcsinh(1) \).
By the collar lemma,  \( E(j,h,2m) \) has exactly two closed geodesics of length less than \( 2\arcsinh(1) \) in its interior, namely the interior cuffs of the given pants decomposition. 

We note that the difference in the lengths of the boundary curves of the edge surfaces is what allows us to encode the direction of the edges in the Cayley graph for $G$ into the geometry of the surface \( X_S^G \) that we will now construct.
For all \( g \in G \), let \( V_g \) be a copy of \( V \).
Also, for all \( g, h \in G \) and for each \( (j,m) \in J \times \bn \), take a copy of \( E(j,h,2m) \) and identify the boundary component of length \( \lambda(j,h,2m) \) to \( V_g \) along \( \partial (j,h,2m) \)  and the other boundary component to \( V_{gh} \) along \( \partial(j,h,2m-1) \) via orientation-reversing isometries.
We define \( X \) to be the resulting surface and, by the definition of the gluings, every point has an open neighborhood isometric to an open subset of \( \bh \); hence, \( X \) is a hyperbolic surface. 
Let \( \mathcal P \) denote the pants decomposition of \( X \) composed of all closed geodesics of length strictly less than \( 2\arcsinh(1) \). 
Lastly,  we modify the twist parameters of \( X \) associated to \(\mathcal P\) in such a way that there is an action of \( G \) on the resulting hyperbolic surface \( X_S^G \) by isometries satisfying \( gV_h = V_{gh} \), for each \(g,h \in G\). 
Therefore, by construction, we can realize \( G \) as a subgroup of \( \isom(X_S^G) \). 

We will now proceed through a sequence of lemmas that will establish the desired geometrical properties of \( X_S^G \) and the topology of \( X_S^G \).

\begin{Lem}
\label{lem:geod-complete}
\( X_S^G \) is a complete hyperbolic surface.
\end{Lem}

\begin{proof}
We have already observed that \( X_S^G \) is a hyperbolic surface.
By construction, \( X_S^G \) has no planar ends and admits a pants decomposition all of whose curves have length bounded from above; hence, completeness follows from Proposition \ref{prop:complete}.
\end{proof}

We remark that Lemma \ref{lem:geod-complete} has the stronger conclusion that \( X_S^G \) is equal to its convex core. 

\begin{Lem}
\label{lem:isom}
The isometry group of \( X_S^G \) is isomorphic  to \( G \). 
\end{Lem}

\begin{proof}
As already noted, by construction, we can view \( G \) as a subgroup of \( \isom(X_S^G) \), so it remains to show that every isometry can be represented by an element of $G$.

Let \( \mathcal V \subset \mathcal P \) denote the elements that correspond to the boundary components of \( V_g \) for each \( g \in G \).
Let \( \gamma \in \mathcal P \) and let \( \tau \) be an isometry of \( X_S^G \).  
Recall that every geodesic in \( \mathcal P \)  has length strictly less than \( 2\arcsinh(1) \).
If \( \tau(\gamma) \) is not in \( \mathcal P \), then it must intersect an element of \( \mathcal P \) and hence \( \ell(\gamma) > 2\arcsinh(1) \) by the collar lemma, a contradiction.

Moreover, as the elements of \( \mathcal V \) are the only elements of \( \mathcal P \) of length strictly less than \( \arcsinh(1) \), we see that \( \tau \) preserves \( \mathcal V \). 
In particular, \( \tau \) must permute the closures of the components of \( X_S^G \ssm \bigcup_{\mu \in \mathcal V} \mu \). Note that these components are precisely the vertex and edge surfaces from the construction of \( X_S^G\).

Now fix \( g \in G \).
Since $\tau$ is an isometry, it sends one of these components to another of the same topological type. 
Thus, there exists \( g_\tau \) such that \( \tau(V_g) = V_{g_\tau} \).
Let \( h = g\cdot g_{\tau}^{-1} \in G \), so then \( h\circ \tau (V_g) = V_g \). 
Restricting \( h \circ \tau \) to \( V_g \), we obtain an isometry of \( V_g \); however, by construction, \( V_g \) has no non-trivial isometries.
This implies that \( h \circ \tau \) restricts to the identity on \( V_g \); in particular, there exists an open subset of \( X_S^G \) on which \( h\circ\tau \) is the identity and hence \( h\circ \tau \) is the identity on all of \( X_S^G \).
We conclude that \ \( \tau = h^{-1} \) and, thus, \( G = \isom(X_S^G) \).
\end{proof}

Note that by construction, the action of \( G \) on \( X_S^G \) is by orientation-preserving isometries, so, in fact, \( \isom(X_S^G) = \isom^+(X_S^G) \). 

\begin{Lem}
\label{lem:covering}
The quotient manifold \( G \backslash X_S^G \) is homeomorphic to \( S \). 
\end{Lem}

\begin{proof}
Let \( E \) be a genus-1 surface with two boundary components.
The surface \( G\backslash X_S^G \) is obtained from the surface \( R_\cE \) by taking a copy of \( E \) and attaching one boundary component to \( \partial(j,h,2m) \) in \(R_\cE\) and the other to \( \partial(j,h, 2m-1) \) in \(R_\cE\) for each \( (j,h,m) \in J\times G \times \bn \). 
It follows that the topological end spaces of \( R_\cE \) and \( G \backslash X_S^G \) are homeomorphic.
Now, recall that \( R_\cE \) and \( S \) have homeomorphic end spaces and observe that \( G \backslash X_S^G \) is borderless and has no planar ends.
The result follows from the classification of surfaces. 
\end{proof}

\begin{Lem}
\label{lem:action-trivial}
If \( G \) is finite, then the action of \( \isom(X_S^G) \) on \( \cE(X_S^G) \) is trivial.
\end{Lem}

\begin{proof}
By Lemma \ref{lem:isom}, we identify \( G \) with \( \isom(X_S^G) \).
Let \( V_h \) be a vertex surface contained in \( X_S^G \).
Fix an exhaustion \( \{ K_n' \}_{n\in\bn} \) of \( V_h \), where \( K_n' \) is a compact, connected, subsurface of \( V_h \) such that each component of \( V_h \ssm K_n' \) is unbounded and \( K_n' \) is contained in the interior of \( K_{n+1}' \) for each \( n \in \bn \). 

Now let \( K_n''  = \bigcup_{g\in G} gK_n' \) and let \( K_n \) be the union of \( K_n'' \) with the bounded components of \( X_S^G \ssm K_n'' \). 
The number of bounded components of \( X_S^G \ssm K_n'' \) is bounded from above by the number of boundary components of \( V_h \) in \( K_n' \) multiplied by the order of \( G \). 
Observe that for large \( n \), the finiteness of \( G \) and the compactness of the edge surfaces guarantee that \( K_n \) is connected.
By forgetting a finite number of sets, we can therefore assume that \( K_n \) is connected for all \( n \).
This implies that \( \{ K_n \}_{n\in\bn} \) is a \( G \)-invariant exhaustion of \( X_S^G \) by compact, connected, subsurfaces such that each component of \( X_S^G \ssm K_n \) is unbounded and \( K_n \) is contained in the interior of \( K_{n+1} \) for each \( n \in \bn \). 

Let \( U \) be a component of the complement of \( K_n \) for some \( n \in \bn \) and let \( g \in G \).
We claim \( g\cdot U = U \). 
To see this, first observe that \( g\cdot U \) is also a component of the complement of \( K_n \) since \( g\cdot K_n = K_n \).
Therefore, we need only show that there exists a path in \( X_S^G \ssm K_n \) connecting \( U \) and \( g\cdot U \). 
Fix \( s \in G \). Then, by construction,  \( U \cap V_s \) is unbounded; in particular, there exists some \( (j, t) \in J\times G \)  and some \( N \in \bn \) such that the boundary components of \( V_s \) labelled \( \partial(j,t,m) \) with \( m > N \) are contained in \( U \). 
Choose \( M \in \bn \) such that \( 2M-1 > N \). For every \( m > M \), the boundary component \( \partial(j,t,2m) \) of \( V_s \) is contained in \( U \), the boundary component \( \partial(j,t,2m-1) \) of \( V_{gs} \) is contained in \( gU \), and these components bound a copy of \( E(j,h,2m) \). 
By construction, this copy of \( E(j,h,2m) \) is not a bounded component of \( X_S^G \ssm K_n'' \) and is, therefore, in the complement of \( K_n \). Hence, there is a path from \( U \) to \( gU \) in the complement of \( K_n \) and, thus, \( gU = U \).

Given an end \( e \) of \( X_S^G \), we can identity \( e \) with a nested chain \( U_1 \supset U_2 \supset \cdots \), where \( U_n \) is a connected component of \( X_S^G \ssm K_n \). 
As we just saw, \( U_n \) is \( G \)-invariant for each \( n \in \bn \),  and hence each element of \( G \) fixes \( e \). 
\end{proof}

\begin{Lem}\label{lem:surface}
If \( G \) is finite, then \( X_S^G \) is homeomorphic to \( S \). 
\end{Lem}

\begin{proof}
First observe that \( X_S^G \) has no planar ends and is borderless.
Using Lemma~\ref{lem:covering}, the action of \( G \) on \( X^G_S \) yields a covering map \( \pi \co X_S^G \to S \).
As \( G \) is finite, \( \pi \) is proper  and hence, by Proposition \ref{prop:end-map}, induces a continuous map \( \widehat \pi \co \Ends(X_S^G) \to \cE \). 

The surjectivity of \( \pi \) guarantees that \( \widehat \pi \) is surjective.
The injectivity of \( \widehat \pi \) follows from Lemma~\ref{lem:action-trivial}: 
if \( e \) and \( e' \) are distinct ends of \( X_S^G \), then, by Lemma~\ref{lem:action-trivial}, they have \( G \)-invariant disjoint neighborhoods.
As \( \pi \) is an open map, the images of these neighborhoods give disjoint neighborhoods of \( \widehat \pi(e) \) and \( \widehat \pi(e') \). 
Since \( \widehat \pi \) is a continuous bijection between compact Hausdorff spaces, it is a homeomorphism; hence, by the classification of surfaces, \( X_S^G \) and \( S \) are homeomorphic. 
\end{proof}

Combining Lemma~\ref{lem:geod-complete}, Lemma~\ref{lem:isom}, and Lemma~\ref{lem:surface} yields:
\begin{Thm}
\label{thm:finite}
Let \( S \) be an orientable infinite-genus 2-manifold with no planar ends. 
If \( G \) is a finite group, then \( S \) admits a complete hyperbolic metric whose isometry group is isomorphic to \( G \). \qed
\end{Thm}

Using Lemma~\ref{lem:covering} and Lemma~\ref{lem:surface}, we record a purely topological result. 

\begin{Prop}
\label{prop:finite-cover}
Let \( S \) be an orientable infinite-genus 2-manifold with no planar ends.
If  \( G \) is any finite group, then there exists a regular covering \( \pi \co S \to S \) such that the deck group associated to \( \pi \) is isomorphic to \( G \). \qed
\end{Prop}

We will discuss a strengthening of Proposition~\ref{prop:finite-cover} under an additional assumption on \(\cE\) in Section~\ref{sec:covers}. 

\subsection{The Loch Ness monster surface}
We now explore the case when \( G \) is denumerable\footnote{A set is denumerable if it is both countable and infinite.}.
Recall that the Loch Ness monster surface is the unique---up to homeomorphism---one-ended infinite-genus 2-manifold.
Let \( G, S, \) and \( X_S^G \) be as above.

\begin{Lem}
\label{lem:loch}
If \( G \) is a denumerable group, then \( X_S^G \) is homeomorphic to the Loch Ness monster surface (regardless of the topology of \( S \)).
\end{Lem}

\begin{proof}
By construction, \( X_S^G \) is borderless, has infinite genus, and does not have any planar ends.
It is only left to show that \( X_S^G \) is one ended.
In order to do this we must show that the complement of every compact set in \( X_S^G \) has a unique unbounded connected component.

Fix an enumeration \( \{ g_i \}_{i\in\bn} \) of the elements of \( G \).
Let \( \{K_n'\} \) be an exhaustion of \( V_{g_1} \) by compact, connected subsurfaces such that \( K_n' \) is in the interior of \( K_{n+1}' \) and each component of \( V_{g_1} \ssm K_n' \) is unbounded.
Now, obtain \( K_n'' \) from \( K_n' \) by taking the union of \( K_n' \) with each edge surface whose boundary lives in \( K_n' \).
Finally, define
\[
K_n = \bigcup_{k=1}^n g_k K_n'',
\]
then \( \{K_n\}_{n\in\bn} \) is an exhaustion of \( X_S^G \) by compact, connected subsurfaces such that \( K_n \) is in the interior of \( K_{n+1} \) and each component of \( X_S^G \ssm K_n \) is unbounded.

Let \( U_n \) denote the complement of \( K_n \); we must show that \( U_n \) is connected.
Let \( x,y \in U_n \).
Without loss of generality, we can assume that \( x \in V_{g_s} \) and \( y \in V_{g_t} \) for some \( s, t \in \bn \) (any point in an edge surface can be connected to a vertex surface since every component of \( U_n \) is unbounded and each edge surface is compact). 
 
Note that there exists \( N \in \bn \) such that \( V_{g_N} \subset U_n \).
Now observe that \( K_N \) contains only finitely many of the boundary components of \( V_{g_s} \) and \( V_{g_t} \), but there exists infinitely many edge surfaces connecting \( V_{g_N} \) to each of \( V_{g_s} \) and \( V_{g_t} \). 
Thus, there exist paths from \( x \) to \( V_{g_N} \) and \( y \) to \( V_{g_N} \) and, since \( V_{g_N} \) is connected and contained in \( U_n \), we can find a path in \( U_n \) from \( x \) to \( y \). Hence, \( U_n \) is connected.
\end{proof}

Combining Lemma~\ref{lem:geod-complete}, Lemma~\ref{lem:isom}, and Lemma~\ref{lem:loch} yields:

\begin{Thm}
\label{thm:lochness}
Let \( L \) be the Loch Ness monster surface and let \( G \) be a group.
There exists a complete hyperbolic metric on \( L \) whose isometry group is isomorphic to \( G \) if and only if \( G \) is countable.
\qed
\end{Thm}


\section{Denumerable groups} \label{Denumerable}

Given a denumerable group \( G \) and an orientable infinite-genus 2-manifold \( S \), the goal of this section is to provide topological conditions on \( S \) guaranteeing the existence or non-existence of a complete hyperbolic metric on \( S \) whose isometry group is isomorphic to \( G \).
We break the section into three subsections covering three topological possibilities for \( S \).
We will later argue, in Section~\ref{sec:trichotomy}, that this covers all possibilities for infinite-genus 2-manifolds with no planar ends.

\subsection{Non-displaceable subsurfaces}

Following Mann and Rafi \cite{MannRafi}, given a surface \( S \) and a subset \( U \) of \( \Homeo(S) \), a subset \( \Sigma \) of \( S \) is \emph{non-displaceable with respect to $U$} if \( f(\Sigma) \cap \Sigma \neq \varnothing \) for every \( f \in U \).  We say \(\Sigma\) is \emph{non-displaceable} if it is non-displaceable with respect to all of \(\Homeo(S)\).

For examples, consider:

\begin{Lem} \label{lem:nd-countable}

Let \( S \) be an orientable surface with countable end space \( \cE \).
If the Cantor--Bendixson degree of \( \cE \) is at least 3, then \( S \) contains a compact non-displaceable subsurface.
\end{Lem}

\begin{proof}
Let \( \al +1 \) and \( d \) be the Cantor--Bendixson rank and degree of \( \cE \), respectively.
Fix a decomposition \( \cE = E_1 \sqcup \cdots \sqcup E_d \) such that, for each \( k \in \{1, \ldots, d\} \), \( E_k \) is clopen and contains exactly one point of \( \cE^\al \).
It follows that we can find pairwise-disjoint separating simple closed curves \( c_1, \ldots, c_d \) such that one component of \( S \ssm c_k \) induces the open set \( E_k \) in \( \cE \), and
 such that there must be a unique compact component \( K \) of \( S \ssm \bigcup_{k=1}^d c_k \).
We claim that the surface \( K \) is non-displaceable: indeed, if there is a self-homeomorphism \( f \) of \( S \) such that \( f(K) \cap K = \varnothing \), then there must be a component of \( S \ssm f(K) \) containing more than a single element of \( \cE^{\al} \) since the Cantor--Bendixson degree of \(\cE\) is at least 3. This yields a contradiction since \(K\) and \(f(K)\) must induce the same partition of \(\cE\). 
\end{proof}

In several places throughout the paper, we will cite the existence of a compact non-displaceable surface without proof since the argument would be nearly identical to the one in the proof of Lemma~\ref{lem:nd-countable}. 

We record the following lemma that is an immediate corollary of Lemma~\ref{lem:prop-discont}.
\begin{Lem} 
\label{lem:finite-isom}
Let \( X \) be a complete hyperbolic surface with non-abelian fundamental group and let \( \Lambda \subseteq \isom(X) \).
If \( X \) contains a compact non-displaceable subset with respect to \( \Lambda \), then \( \Lambda \) is finite. \qed
\end{Lem}

Combining Theorem~\ref{thm:finite} and Lemma~\ref{lem:finite-isom}, we obtain:

\begin{Thm}
\label{thm:lack}
Let \( S \) be an orientable infinite-genus 2-manifold with no planar ends containing a compact non-displaceable subsurface.
Given a group \( G \), there exists a complete hyperbolic metric on \( S \) whose isometry group is isomorphic to \( G \) if and only if \( G \) is finite.
\qed
\end{Thm}

Theorem~\ref{thm:lack} together with Lemma~\ref{lem:nd-countable} yields:

\begin{Cor}\label{cor:CB1}
Let \( S \) be an orientable infinite-genus 2-manifold with no planar ends and whose space of ends is countable of Cantor--Bendixson rank at least 3.
Given a group \( G \), there exists a complete hyperbolic metric on \( S \) whose isometry group is isomorphic to \( G \) if and only if \( G \) is finite. \qed
\end{Cor}

Before continuing, we record a proposition about finite-genus surfaces that we will use later:

\begin{Prop}
\label{prop:finite-genus}
Every hyperbolic surface of positive finite genus has finite isometry group. 
\end{Prop}

\begin{proof}
Let \( X \) be a hyperbolic surface of positive finite genus \( g \).
We can then choose a simple closed geodesic \( \gamma \) bounding a compact subsurface \( \Sigma \) of genus \( g \).
As \( \Sigma \) has the same genus of \( X \), it must be that \( \Sigma \) is non-displaceable.
Hence, by Lemma \ref{lem:finite-isom}, \( \isom(X) \) is finite.
\end{proof}

\subsection{Radial symmetry}
\label{sec:radial}

We now turn our focus to giving a sufficient condition on a surface for every countable  group to be realizable as an isometry group of some complete hyperbolic metric on the given surface.

\begin{Def}
\label{def:radial-symmetry}
A compact Hausdorff topological space \( E \) has \emph{radial symmetry} if either \( E \) is a singleton or there exists a point \( x \in E \) and a collection of pairwise homeomorphic, non-compact sets \( \{E_n\}_{n\in\bn} \) such that the closure of \( E_n \) is disjoint from \( E_m \) whenever \( n \neq m \) and such that
\[
E \ssm \{x\} = \bigsqcup_{n\in\bn} E_n,
\]
Note by the compactness of \( E \) and the assumptions on \( E_n \), it must be that \( x \) is in the closure of \( E_n \) for every \( n \in \bn \).
We refer to the point \( x \) as a \emph{star point} of \( E \).
\end{Def}

In Section~\ref{sec:SSandPS}, we will show that radial symmetry is equivalent to the notion of self-similarity introduced in \cite{MannRafi}.

For examples, both \( \omega+1 \) (i.e.~a convergent sequence together with its limit point) and the Cantor set have radial symmetry. 
Also, any compact countable Hausdorff space of Cantor--Bendixson degree 1 has radial symmetry (see Lemma~\ref{lem:pointed-degree} below), of which there are uncountably many.

Let \( S \) be an orientable infinite-genus 2-manifold with no planar ends whose end space has radial symmetry.
Given a denumerable group \( G \), we now construct a complete hyperbolic surface \( Y_S^G \) that is homeomorphic to \( S \) and whose isometry group is isomorphic to \( G \).
If the end space of \( S \) has radial symmetry, the cardinality of \( \cE(S) \) is either one or infinite: if \( S \) has a unique end, then \( S \) is the Loch Ness monster surface and, appealing to Theorem~\ref{thm:lochness}, we simply  set \( Y_S^G = X_S^G \).
Therefore, going forward, we assume that \( S \) has infinitely many ends.

The construction will closely follow the construction of \( X_S^G \) in Section~\ref{Allcock}.
We begin identically by letting \( C_G \) be the Cayley graph for \( G \) as described in Section~\ref{Allcock}, so that the group of label- and direction-preserving graph automorphisms of \( C_G \) is isomorphic to \( G \). 

Let \( \cE \) be the end space of \( S \) and let \( x \) be a star point of \( \cE \) so that there exists a decomposition 
\[
\cE \ssm \{x\} = \bigsqcup_{n\in\bn} E_n
\]
as in the definition of radial symmetry.
Let \( E = E_1 \) and \( \bar E = E_1 \cup \{x\} \) and note that \( E_n \) (respectively \( \bar E_n = E_n \cup\{x\} \)) is homeomorphic to \( E \) (respectively \( \bar E \)) for all \( n \in \bn \).

Let us view the 2-sphere as the one-point compactification of \( \br^2 \), written as \( \widehat \br^2 = \br^2 \cup \{\infty\} \). 
Fix an embedding \( \iota \co \bar E \to \widehat \br^2 \) satisfying \( \iota(x) = \infty \) and with \(\iota(E)\) in the lower half-plane.  

Fix a bijection \( f \co G \to \bn \) and let
\[
Z = \br^2 \ssm \left( \iota(E) \cup  \bigcup_{(h,m)\in G \times\bn} B \left((f(h),m), \frac14\right) \right),
\]
where \( B(y,r) \) is the (open) Euclidean ball in \( \br^2 \) of radius \( r \) and center \( y \). 
Now, let  \( R \) be the bordered infinite-genus surface with no planar ends obtained by taking the connected sum of \( Z \) with countably many tori.
Note that the requirement that \( R \) has no planar ends fixes the topology of \( R \) up to homeomorphism.

From here on, the construction of \( Y_S^G \) follows---with minor adjustments---that of \( X_S^G \); we provide the details for completeness.
First note that, arguing as in the proof of Lemma~\ref{lem:end-space}, we see that the end space of \( R \) is homeomorphic to \(  \bar E \). 
Also note that the boundary components of \( R \) are in bijection with \( G \times \bn \). Let \( \partial(h,m) \) denote the boundary component bounding the ball with center \( (f(h),m) \in \br^2 \).
Choose an injective function \( \lambda \co G \times \bn \to (0, \arcsinh(1)) \subset \br \). 
Let \( V \) be a hyperbolic surface with totally geodesic boundary homeomorphic to \( R \) such that \( V \) has no isometries, \( V \) admits a pants decomposition with all interior cuff lengths between \( \arcsinh(1) \) and \( 2\arcsinh(1) \), and the length of \( \partial(h,m) \) is \( \lambda(h,m) \). 
Again, as before, let \( E(h,2m) \) be a hyperbolic surface homeomorphic to a torus with two boundary components obtained by gluing together pairs of pants, one with boundary lengths \( \lambda(h,2m), \arcsinh(1) \) and \( \arcsinh(1) \), and the other with lengths \( \lambda(h,2m-1), \arcsinh(1), \) and \( \arcsinh(1) \). 

We now construct \( Y_S^G \).
Let \( V_g \) be a copy of \( V \) for each \( g \in G \).
For all \( g, h \in G \) and for each \( m \in \bn \) take a copy of \( E(h,2m) \) and attach one boundary component of \( V_g \) along \( \partial (h,2m) \) and the other to \( V_{gh} \) along \( \partial(h, 2m-1) \) via orientation-reversing isometries.
In addition, we require the attaching maps to be compatible with the action of \( G \) on \( C_G \).
We define \( Y_S^G \) to be the resulting surface.
Note that by construction, \( Y_S^G \) is orientable, borderless, infinite genus and has no planar ends.

Arguing as in Lemma~\ref{lem:geod-complete} and Lemma~\ref{lem:isom}, we have:

\begin{Lem}
\label{lem:Y-isom}
\( Y_S^G \) is a complete hyperbolic surface whose isometry group is isomorphic to \( G \). 
\qed
\end{Lem}

It is left to show that \( Y_S^G \) is homeomorphic to \( S \). 
As \( V_g \) is a closed subset of \( Y_S^G \), the inclusion \( V_g \hookrightarrow Y_S^G \) is proper and induces a map \( \vp_g \co \Ends(V_g) \to \Ends(Y_S^G) \); moreover, as any two ends of \( V_g \) can be separated by a separating simple closed curve whose image in \( Y_S^G \) is also a separating simple closed curve, we see that \( \vp_g \) is injective. 
As \( V_g \) is homeomorphic to \( R \), we can identify \( \Ends(V_g) \) with \( \bar E \) and view \( \vp_g \co \bar E \to \Ends(Y_S^G) \). 

Now for any end \( e \) of \( V_g \) corresponding to a point in \( E \), there exists a simple separating closed curve in \( V_g \) such that one component of the complement contains \( e \) and is disjoint from \( \partial V_g \).
This guarantees that \( \vp_g(E) \) and \( \vp_h(E) \) are separated---that is, each set is disjoint from the other's closure---in \( \Ends(Y_S^G) \) for all distinct \( g, h \in G \). 
However, arguing as in Lemma~\ref{lem:loch}, we see that \( \vp_g(x) = \vp_h(x) \) for all \( g,h \in G \), where \(x\) is the star point. 
Using a compact exhaustion of \( Y_S^G \) analogous to that constructed in Lemma~\ref{lem:loch} yields 
\[
\Ends(Y_S^G) = \bigcup_{g\in G} \vp_g(\bar E).
\]
Let \( y = \vp_g(x) \) for any \( g \in G \), which we emphasize is well-defined.
We have shown that
\[
\Ends(Y_S^G) \ssm \{y\} = \bigsqcup_{g\in G} \vp_g(E).
\]
Moreover,
\[
\bigsqcup_{g\in G} \vp_g(E) \cong \bigsqcup_{n\in\bn} E_n
\]
and \( \Ends(Y_S^G) \) is the one-point compactification of the former and \( \cE \) is the one-point compactification of the latter; hence, \( \Ends(Y_S^G) \) and \( \cE \) are homeomorphic and thus \( S \) and \( Y_S^G \) are homeomorphic, establishing:

\begin{Lem}
\label{lem:Y-homeo}
\( Y_S^G \) is homeomorphic to \( S \).
\qed
\end{Lem}

Combining Theorem~\ref{thm:lochness} (for the one-ended case), Theorem~\ref{thm:finite}, Lemma~\ref{lem:Y-isom}, and Lemma~\ref{lem:Y-homeo}, we obtain:

\begin{Thm}
\label{thm:pointed}
Let  \( S \) be an orientable infinite-genus 2-manifold with no planar ends and whose space of ends has radial symmetry.
Given a group \( G \), there exists a complete hyperbolic metric on \( S \) whose isometry group is isomorphic to \( G \) if and only if \( G \) is countable.
\qed
\end{Thm}

\begin{Rem}
\label{rem:self-similar_planar}
The notion of radial symmetry can be extended to pairs of topological spaces and, in particular, to the pair \( (\cE(S), \cE_{np}(S)) \) associated to \( S \).
In this case, one only needs to additonally require that \( E_n \cap \cE_{np} \) and \( E_m \cap \cE_{np} \) are homeomorphic for all \( n, m \in \bn \).   
With only slight modification to the construction above---namely, in the adding of handles to \( Z \)---the assumption on (the lack of) planar ends is readily removed in Theorem \ref{thm:pointed}.
Note that if a surface is not homeomorphic to the plane, has planar ends, and its end space has radial symmetry, then it necessarily has infinitely many planar ends.
\end{Rem}

There is a nice statement of Theorem~\ref{thm:pointed} when the end space of \( S \) is countable:

\begin{Cor}\label{cor:CB2}
Let \( S \) be an orientable  infinite-genus 2-manifold with no planar ends and whose space of ends is countable of Cantor--Bendixson degree 1.
Given a group \( G \), there exists a complete hyperbolic metric on \( S \) whose isometry group is isomorphic to \( G \) if and only if \( G \) is countable.
\end{Cor}

The above corollary follows immediately from Theorem~\ref{thm:pointed} together with the following lemma:

\begin{Lem}
\label{lem:pointed-degree}
Let \( \cE \) be a compact, countable, Hausdorff space.
The Cantor--Bendixson degree of \( \cE \) is 1 if and only if \( \cE \) has radial symmetry.
\end{Lem}

\begin{proof}
First assume that the Cantor--Bendixson degree of \( \cE \) is 1, so that, by Theorem~\ref{thm:classification}, \( \cE \) is homeomorphic to \( \omega^\al +1 \), where \( \al +1 \) is the rank of \( \cE \).
For each \( n \in \bn \), let \( E_n \) be a copy of \( \omega^\al \) and let
\[
E = \bigsqcup_{n\in\bn} E_n.
\]
Observe that the Cantor--Bendixson rank of \( E \) is \( \al \); hence, the Cantor--Bendixson rank of the one-point compactification \( \bar E = E \cup \{\infty\} \) of \( E \) is \( \al +1 \) and \( \bar E^\al = \{\infty\} \).
By Theorem~\ref{thm:classification}, we have that \( \bar E \) is homeomorphic to \( \cE \); moreover, by construction, \( \bar E \) (and hence \( \cE \)) has radial symmetry. 

Now suppose that the Cantor--Bendixson degree of \( \cE \) is \( d \) and \( d> 1 \) and assume that \( \cE \) has radial symmetry.
Let \(\alpha +1\) be the rank of \(\cE\) and let \( x \in \cE \) such that
\[
\cE \ssm \{x\} = \bigsqcup_{n\in\bn} E_n
\]
as in the definition of radial symmetry.
Necessarily, as \( \cE^\al \) has at least two points, there exists \( n \in \bn \) such that \( E_n \cap \cE^\al \neq \varnothing \). 
However, as \( \cE^\al \) has a finite number of points, there exists \( m \in \bn \) such that \( E_m \cap \cE^\al = \varnothing \); the Cantor--Bendixson ranks of \( E_n \) and \( E_m \) do not agree and they are, therefore, not homeomorphic, a contradiction.
\end{proof}


\subsection{Doubly pointed spaces}

The end space of a surface \( S \) is \emph{doubly pointed} if there are exactly two ends of \( S \) whose orbits under the action of \( \rm{Homeo}(S) \) are finite (and hence of cardinality at most two).
For example, any infinite-genus surface with no planar ends and countable end space of Cantor--Bendixson degree 2 is doubly pointed.

Before stating the main theorem of the section, recall that a group is \emph{virtually cyclic} if it contains a cyclic subgroup of finite index.
Note that every finite group is virtually cyclic.
The goal of this section is to prove the following:

\begin{Thm}
\label{thm:doubly-pointed}
Let \( S \) be an orientable 2-manifold with non-abelian fundamental group and a doubly pointed end space.
The isometry group of any complete hyperbolic metric on \( S \) is virtually cyclic.
\end{Thm}

Since every finite group is virtually cyclic, we focus on denumerable groups.

\begin{Lem}
\label{lem:closure}
Let \( X \) be a complete hyperbolic surface with a doubly pointed end space and infinite isometry group, and let $\Lambda \subset \mbox{Isom}(X)$ be an infinite subset. If \( e_1 \) and \( e_2 \) are the unique ends with finite orbit under \( \rm{Homeo(X)} \), then the closure of any infinite \( \Lambda \)-orbit contains \( \{e_1,e_2\} \).
\end{Lem}

\begin{proof}
First suppose there exists an end \( e \) of \( X \) such that the closure of \(\Lambda \cdot e \) is infinite and disjoint from \( \{e_1, e_2\} \). 
In this case, there exists a compact surface whose components separate the sets \( \{e_1\}, \{e_2\}, \) and \( \Lambda\cdot e \) and hence is non-displaceable with respect to \( \Lambda \). 
Therefore, by Lemma~\ref{lem:finite-isom}, \( \Lambda \) must be finite, a contradiction.

Next, suppose there exists an end \( e \) with infinite orbit such that the closure of \( \Lambda \cdot e \) contains \( e_1 \) but not \( e_2 \). 
First, since \( \Lambda \) is countable, there exists an element \( e' \in \Lambda \cdot e \) that is an isolated point of the orbit \( \Lambda \cdot e \) equipped with the subspace topology inherited from \( \cE(X) \). 
Now, let \( \gamma \) be a separating simple closed geodesic in \( X \) such that the components of \( X \ssm \gamma \) induce a partition of \( \cE(X) \) into two sets, \( U_1 \) and \( U_2 \), with  \( e_1 \in U_1, e_2 \in U_2 \), and \( \Lambda\cdot e \cap U_2 = \{e'\} \). 
Note that by the assumption on the closure of \( \Lambda \cdot e \), it must be that \( \widehat f(e_1) = e_1 \) and \( \widehat f(e_2) = e_2 \) for every \( f \in \Lambda \). 
Let \( \Lambda' \subset \Lambda \) be the elements of \( \Lambda \) that are not in the stabilizer of \( e' \): the infiniteness of \( \Lambda\cdot e \) guarantees that \( \Lambda' \) is infinite.  
Now, for every \( f \in \Lambda' \), neither \( \widehat f(U_1) \subset U_1 \) nor \( \widehat f(U_1) \subset U_2 \); hence, \( f(\gamma) \cap \gamma \neq \varnothing \).
Therefore, \( \gamma \) is non-displaceable for \( \Lambda' \) and thus, by Lemma \ref{lem:finite-isom}, \( \Lambda' \) is finite, a contradiction. 
\end{proof}

\begin{proof}[Proof of Theorem~\ref{thm:doubly-pointed}]

Let \( X \) be \( S \) equipped with a complete hyperbolic metric.  Since every finite group is virtually cyclic, we assume that \(\isom(X)\) is infinite.  It then suffices to construct a homomorphism from a finite-index subgroup of $\isom(X)$ to \(\mathbb{Z}\) with finite kernel. Let \( e_1 \) and \( e_2 \) be the unique ends of \( X \) with finite orbits, and let \( \Lambda \) denote the stabilizer of \( e_1 \) in \( \isom(X) \), which is index at most two.

First assume that \( X \) has infinitely many ends.
Choose an end \( e \) with infinite orbit and note that the closure of this orbit contains both \( e_1 \) and \( e_2 \) by Lemma~\ref{lem:closure}.
Fix an oriented simple closed curve \( \gamma \) separating \( e_1 \) and \( e_2 \). 
Let \( L_\gamma \) denote the set of ends to the left of \( \gamma \) and contained in \( \Lambda\cdot e \), and assume \( e_1 \) is to the left of \( \gamma \).
We define \( \vp \co \Lambda \to \bz \) by 
\[
\vp(g) = |\widehat g(L_\gamma) \ssm L_\gamma| - |L_\gamma \ssm \widehat g(L_\gamma)|.
\]

Note that \( |\widehat g(L_\gamma) \ssm L_\gamma| \) measures the size of the subset of ends in \( L_\gamma \) that \( g \) moves to the right of \( \gamma \). On the other hand, \( |L_\gamma \ssm \widehat g(L_\gamma ) | \) measures the size of the subset of \( L_\gamma \) that is not in the image of \(L_\gamma \) under \( g \), and therefore which \textit{is} in the image of the set of ends in the orbit of \( e \) which are to the right of \( \gamma \). Thus, \( \vp(g) \) all together measures the net change of how many ends in \( \isom(X) \cdot e\) move from the left of \( \gamma \) to the right of \( \gamma \). 

We claim that both \( |\widehat g(L_\gamma) \ssm L_\gamma | \) and \( |L_\gamma \ssm \widehat g(L_\gamma)| \) are finite, and that, therefore, \( \vp \) is well-defined. 

To see this, assume that \( |\widehat g(L_\gamma) \ssm L_\gamma | \) is infinite and let  \( A \subset L_{\gamma} \) denote the infinitely many ends in \(\Lambda \cdot e\) that move from the left of \(\gamma \) to the right of \(\gamma \). Each such end corresponds to a distinct isometry, and therefore associated to \( A \) is an infinite subset \( U \subset \Lambda \). Then by Lemma~\ref{lem:closure}, \( e_1 \) must be in the closure of \( A \). Thus by continuity, \(g\) must move \( e_1 \) to the right of \( \gamma \), contradicting our assumption that \( \Lambda \) fixes \( e_1 \). Replacing \(e _1 \) with \( e_2 \) and applying the same argument to the ends in the orbit of \( e \) and to the right of \( \gamma \) implies that  \( |L_\gamma \ssm \widehat g(L_\gamma)| \) is finite as well. 

Since \( \vp \) is well-defined, it follows immediately from the fact that  \( \vp \) merely counts the net number of ends in \( \Lambda \cdot e \) that move from the left of \( \gamma \) to the right of \(\gamma\), that it is a homomorphism: the net increase of ends in \( \Lambda \cdot e\) to the right of \(\gamma \) after applying the composition \( g_2 g_1 \) is the sum of the individual net increases.

We will now argue that \( \ker \vp \) is finite.  
Let \( R_\gamma = \Lambda\cdot e \ssm L_\gamma \).
 Let \( \eta \) be a separating simple closed geodesic on \( X \) such that the two connected components \( U_1 \) and \( U_2 \) of \( X \ssm \eta \) satisfy \( e_1, e_2 \in \widehat U_1 \) and such that \( \widehat U_2 \cap L_\gamma \neq \varnothing \) and \( \widehat U_2 \cap R_\gamma \neq \varnothing \).
Observe that, for \( g \in \Lambda \), it is impossible for \( g(\gamma) \subset U_2 \) since \( e_1 \) and \( e_2 \) must be on opposite sides of \(\gamma\), and therefore of \( g(\gamma) \).
Now, if \( g \in \ker\vp \), we claim that \(g(\gamma) \cap \eta \neq \varnothing\) or  \(g(\gamma) \cap \gamma \neq \varnothing\). 
To see this, assume that \(g(\gamma) \cap \eta = \varnothing\). Then the above argument shows that \( g(\gamma) \subset U_1 \). 
If \(\gamma \cap g(\gamma) = \varnothing\) as well, then, by the choice of \( \eta \), either \(\widehat g (L_\gamma) \subseteq L_\gamma\) or \(\widehat g (R_\gamma) \subseteq R_\gamma\); moreover, since \(L_\gamma \cap \widehat U_2 \neq \varnothing\) and  \(R_\gamma \cap \widehat U_2 \neq \varnothing\), the containment must be proper, contradicting the fact that \(\vp(g) = 0\). 
Therefore, if \( K = \gamma\cup\eta \), then \( g(K) \cap K \neq\varnothing \) for every \( g \in \ker\vp \).
In particular, \( K \) is a compact non-displaceable subset of \( X \) with respect to the action of \( \ker\vp \) on \( X \), and hence, by Lemma~\ref{lem:finite-isom}, \( \ker \vp \) is finite.

Now suppose that \( X \) has exactly two ends.
The condition on the fundamental group of \( X \) implies that \( X \) has positive genus and, by Proposition \ref{prop:finite-genus}, the assumption that \( \isom(X) \) is infinite implies \( X \) is infinite genus.
Now if \( X \) contains a planar end, then any non-compact finite-genus subsurface that is a closed subset of \( X \) is non-displaceable, and such a surface contains a compact non-displaceable subsurface of \( X \) (obtained by removing a neighborhood of the planar end).
In this case, \( \isom(X) \) would have to be finite, which is a contradiction; hence, both ends of \( X \) must be non-planar.
From \cite{AramayonaFirst}, there exists a surjective homomorphism \( \vp \co \Lambda \to \bz \) such that if \( g \in \ker \vp \) then \( g(\gamma) \cap \gamma \neq \varnothing \) for any simple closed curve separating the ends of \( X \)---the homomorphism is constructed in \cite[Proposition 3.3]{AramayonaFirst} and the condition on the intersection is \cite[Lemma 3.1]{AramayonaFirst}.
Again, by Lemma \ref{lem:finite-isom}, \( \ker \vp \) is finite.
\end{proof}

As in the previous cases, we can restate Theorem~\ref{thm:doubly-pointed} in terms of the Cantor--Bendixson rank and degree when the end space is countable. Moreover, in the case of countably many ends, we can strengthen Theorem~\ref{thm:doubly-pointed} by giving a partial converse:

\begin{Thm}
\label{thm:deg2}
Let \( S \) be an orientable infinite-genus 2-manifold with no planar ends and whose space of ends is countable of Cantor--Bendixson degree 2 and rank \( \al +1 \).
Let \( G \) be any group.
\begin{enumerate}[(i)]
\item
If \( \al \) is a successor, then there exists a complete hyperbolic metric on \( S \) whose isometry group is isomorphic to \( G \) if and only if \( G \) is virtually cyclic.

\item
If \( \al \) is a limit ordinal, then there exists a complete hyperbolic metric on \( S \) whose isometry group is isomorphic to \( G \) if and only if \( G \) is finite.
\end{enumerate}
\end{Thm}

The end space of a surface is doubly pointed and countable if and only if the Cantor--Bendixson degree of the end space is two.
So, let \( S \) be as in Theorem~\ref{thm:deg2}, then, by Theorem~\ref{thm:finite},  every finite group appears as the isometry group of some complete hyperbolic metric on \( S \); moreover, by Theorem~\ref{thm:doubly-pointed}, the isometry group of a complete hyperbolic metric on \( S \) must be virtually cyclic.
So we need only show that (i) given an infinite virtually cyclic group, it appears as the isometry group when \( \al \) is a successor and (ii) if \( \al \) is a limit ordinal, then the isometry group is finite.
We break these cases into two lemmas.

\begin{Lem}
Let \( X \) be an orientable complete hyperbolic surface whose space of ends \( \cE \) is countable of Cantor--Bendixson degree 2 and rank \( \al +1 \).
If \( \al \) is a limit ordinal, then \( \isom(X) \) is finite.
\end{Lem}

\begin{proof}
Label the two points in \( \cE^\al \) by \( e^+ \) and \( e^- \).
As \( \isom(X) \) is virtually cyclic by Theorem~\ref{thm:doubly-pointed}, if \( \isom(X) \) is infinite, then it must contain an infinite-order element \( f \) (generating a cyclic, finite-index subgroup of \( \isom(X) \)).
By possibly replacing \( f \) by a power of \( f\), we may assume that \( f(e^\pm) = e^\pm \) and that \(f\) is orientation-preserving. 

Choose a separating simple closed geodesic \( \gamma \) in \( X \) that separates the two points of \( \cE^\al \). 
First observe that given any compact \( K \subset X \), there exists \( N \in \bn \) such that \( f^{\pm n}(\gamma) \cap K =\varnothing \) for all \( n > N \): indeed, if not, then \( K \) would be non-displaceable for the action of \( \langle f \rangle \), and hence \( \langle f \rangle \) is finite by Lemma \ref{lem:finite-isom}, a contradiction.
By potentially replacing \( f \) with a power, we may assume that both \( f(\gamma) \) and \( f^{-1}(\gamma) \) are disjoint from \( \gamma \).

Let \( U_+ \) and \( U_- \) denote the components of \( X \ssm \gamma \), labelled such that \( \widehat U_\pm \) is a neighborhood of \( e^\pm \).
Up to relabelling, we may assume that \( f(\gamma) \subset U_+ \).
Observe that as \(f\) is orientation-preserving and \( f^{-1}(\gamma) \cap \gamma = \varnothing \), we must have that \( f^{-1}(\gamma) \subset U_- \).

Now, for \( n \in \bz \), let \( V_n \) denote the subsurface co-bounded by \( f^n(\gamma) \) and \( f^{n+1}(\gamma) \). 
The complement of \( f^n(\gamma) \cup f^{n+1}(\gamma) \) in \(X\) has three components, which partition \(\cE\) into clopen subsets; let $E_n$ denote the \(\widehat{V}_n\) (which is possibly empty).
As noted above, the set of geodesics \( \{f^n(\gamma)\}_{n\in\bn} \) escape every compact set; hence, \( X = \bigcup_{n\in\bz} V_n \).
It follows that
\[
\cE \ssm \cE^\al = \bigsqcup_{n\in\bz} E_n,
\]
for if the complement of the union of the \( E_n \) contained more than two ends, then there would be a compact set--- such as a pair of pants separating \( e_+, e_- \), and the other ends in the complement of \( \bigcup_{n\in\bz} E_n \)--- that infinitely many of the \( f^n(\gamma) \) would intersect non-trivially. 

Now, as \( E_n \) and \( E_m \) are homeomorphic for all \( n, m \in \bz \), if the Cantor--Bendixson rank of \( E_n \) is \( \be \), then it follows that the rank of \( \cE \ssm \cE^\al \) is also \( \beta \); hence, we may conclude that \( \al = \beta+1 \), which contradicts \( \al \) being a limit ordinal.
\end{proof}

\begin{Lem}
Let \( S \)  be an orientable  infinite-genus 2-manifold with no planar ends and whose space of ends \(\cE\) is countable of Cantor--Bendixson degree 2 and rank \( \al + 1 \).
If \( \al \) is a successor and \( G \) is an infinite virtually cyclic group, then \( S \) admits a complete hyperbolic metric  whose isometry group is isomorphic to \( G \).
\end{Lem}

\begin{proof}
Virtually cyclic groups are finitely generated. Fix a finite (symmetric) generating set \( \mathscr S \).
Let \( \Gamma \) be the directed and labelled Cayley graph of \( G \) associated to \( \mathscr S \), that is, \( \Gamma \) is the graph whose vertices correspond to the elements of \( G \) and for each ordered pair \( (g,s) \in G \times \mathscr S \) there is a directed edge from \( g \) to \( gs \) labelled `\( s \)'.
The group of automorphisms of \( \Gamma \) preserving the direction and labels of edges is \( G \).

Note that as \( G \) is an infinite virtually cyclic group, it contains an infinite cyclic normal subgroup \( H \) of finite index.
In particular, \( H \) acts on \( \Gamma \) properly discontinuously and cocompactly and, since \( H \) also acts properly discontinuously and cocomapctly on \( \br \) (which is 2-ended), by Proposition \ref{prop:end-group}, \( \Gamma \) is 2-ended.

Recall that by Theorem~\ref{thm:classification}, \( \cE \) is homeomorphic to \( \omega^\al\cdot2+1 \).
As \( \al \) is a successor, we can define \( W = \omega^{\al-1} + 1 \). 
Note that the one-point compactification of \( \bigsqcup_{n\in\bn} W \) is homeomorphic to \( \omega^\al + 1 \) (again, by Theorem~\ref{thm:classification}).

As in our previous constructions, we need to construct vertex and edge surfaces to build our desired hyperbolic surface.
We first focus on the vertex surface \( V \).
We require the topology of \( V \) to have the following properties, which, by the classification of surfaces, determine the homeomorphism type of \( V \):
\begin{itemize}
\item \( V \) is orientable, has infinite genus, and every end is non-planar,
\item \( \cE(V) \) is homeomorphic to \( W \), and
\item \(\partial V \) is compact with \( 2|\mathscr S| \) connected components.
\end{itemize}
Fix a labelling of the boundary components of \( V \) by \( \mathscr S \times \{0,1\} \); we write \( \partial_s^i \) for the component corresponding to \( (s,i) \in \mathscr S \times \{0,1\} \).
Next, we consider the geometry of \( V \).
Fix an injection  \( \zeta\co \mathscr S \times\{0,1\} \to ( 0, \arcsinh(1) ) \).
For the geometry of \( V \), we require that \( V \) is a hyperbolic surface with totally geodesic boundary satisfying:
\begin{itemize}
\item \( V \) has no isometries,
\item \( \ell(\partial_s^i)  = \zeta(s,i) \) for \( (s,i) \in \mathscr S \times \{0,1\} \), and
\item \( V \) admits a pants decomposition with all interior cuffs having lengths contained in the interval \( (\arcsinh(1), 2\arcsinh(1)) \).
\end{itemize}

As in our previous constructions, we let the edge surface \( E(g,s) \) be a hyperbolic surface homeomorphic to a torus with two boundary components, obtained by gluing two pairs of pants, one with boundary lengths \( \zeta(s,0), \arcsinh(1) \) and \( \arcsinh(1) \), and the other with boundary lengths \( \zeta(s,1), \arcsinh(1), \) and \( \arcsinh(1) \). 

Now, let \( X_\Gamma \) be a hyperbolic surface obtained as follows:
for each \( g \in G \), let \( V_g \) be a copy of \( V \).
For each \((g,s) \in G \times \mathscr S \) take a copy of \( E(g,s) \) and attach one boundary component to \( V_g \) along \( \partial(s,0) \) and the other to \( V_{gs} \) along \( \partial(s,1) \).
In addition, we require these attaching maps to be compatible with the action of \( G \) on \( \Gamma \). 

Arguing as in Lemma~\ref{lem:geod-complete} and Lemma~\ref{lem:isom}, we see that \( X_\Gamma \) is a complete hyperbolic surface and \( \isom(X_\Gamma) \cong G \). 

It is left to show that \( X_\Gamma \) is homeomorphic to \( S \). 
It is clear that \( X_\Gamma \) is borderless, has infinite genus, and has no planar ends, so we need only show that \( \Ends(X_\Gamma) \) is homeomorphic to \( \cE \). 
First observe that \( V_g \) is a closed subspace of \( X_\Gamma \) for each \( g \in G \) and hence there is an end map \( \Ends(V_g) \to \Ends(X_\Gamma) \) induced by the inclusion \( V_g \hookrightarrow X_\Gamma \). 
As \( \partial V_g \) is compact and separating, it follows that \( \Ends(V_g) \) embeds into \( \Ends(X_\Gamma) \); in fact, the same reasoning implies that \( \bigsqcup_{g\in G} \Ends(V_g) \) embeds in \( \Ends(X_\Gamma) \) as a dense subset. 
Recall that \(\cE(V_g)\) is homeomorphic to $W = \omega^{\alpha-1} + 1\) so that, by Theorem~\ref{thm:classification}, it is only left to show that
\[
\Ends(X_\Gamma) \ssm \left(\bigsqcup_{g\in G} \Ends(V_g) \right)
\]
consists of two points. 

For each \( g \in G \), let \( \gamma_g \subset V_g \) be a separating (in $X_\Gamma$) simple closed geodesic inducing a partition of \( \Ends(X_\Gamma) \) into two clopen sets, one of which is \( \Ends(V_g) \), and such that \( h(\gamma_g)=\gamma_{hg} \) for all \( g,h \in G \).
Let \( Y \) be the unique component of \( X_\Gamma \ssm \bigcup_{g\in G} \gamma_g \) with disconnected boundary; it follows that 
\[
\Ends(Y) = \Ends(X_\Gamma) \ssm \left(\bigsqcup_{g\in G} \Ends(V_g) \right).
\]
Now observe that \( G \) acts properly discontinuously and cocompactly on \( Y \); hence, by Proposition \ref{prop:end-group}, \( \cE(Y) \) and \( \cE(\Gamma) \) are homeomorphic.
As \( |\Ends(\Gamma)| = 2 \),  we can then conclude that the Cantor-Bendixson rank and degree for  \( X_\Gamma \) are $\alpha +1$ and $2$, respectively, so that \( X_\Gamma \) is homeomorphic to \( S \).
\end{proof}

Combining Corollaries~\ref{cor:CB1} and \ref{cor:CB2} and Theorem~\ref{thm:deg2} completes the proof Theorem~\ref{main}.

\begin{Thm} \label{main}
Let \( S \) be an orientable infinite-genus 2-manifold with a countable space of ends, none of which are planar, and let \( G \) be an arbitrary group.
Suppose the end space \( \cE \) of \( S \) has characteristic system \( (\al, n) \).
\begin{enumerate}
\item 
If \( n = 1 \), there exists a  complete hyperbolic metric on \( S \) whose isometry group is \( G \) if and only if \( G \) is countable.
\item 
If \( n = 2 \) and $\alpha$ is a successor ordinal, there exists a  complete hyperbolic metric on \( S \) whose isometry group is \( G \) if and only if \( G \) is virtually cyclic.
\item
If \( n \geq 3 \), or $n=2$ and $\alpha$ is a limit ordinal, there exists a complete hyperbolic metric on \( S \) whose isometry group is \( G \) if and only if \( G \) is finite. \qed
\end{enumerate}
\end{Thm}

\section{Radial symmetry and self-similarity}\label{sec:SSandPS}

The goal of this section is to give another characterization of radial symmetry for end spaces of surfaces, which will later allow us to conclude---see Theorems~\ref{thm:partialclass} and \ref{mainUncountable}---that the various results from Section~\ref{Denumerable} cover all orientable infinite-genus 2-manifolds with no planar ends.
In particular, we will show that radial symmetry is equivalent to the notion of self-similarity introduced in \cite{MannRafi}.

Throughout the entirety of the section, we let \( \cE \) denote the end space of an orientable 2-manifold \( S \) and let \( \cE_{np} \subset \cE \) be the closed subset consisting of non-planar ends. 
In order to not burden the reader with notation, we view every subset \( A \) of \( \cE \) as a pair of topological spaces, namely \( (A, A\cap \cE_{np}) \), and will assume all homeomorphisms below are as pairs of nested topological spaces; for instance, if we say \( A \) and \( B \) are homeomorphic subsets of \( \cE \), we mean that there exists a homeomorphism from \( A \) to \( B \) and mapping \( A\cap \cE_{np} \) onto \( B \cap \cE_{np} \). 
(These are the type of homeomorphisms induced by homeomorphisms of the underlying surface.)
The definition of radial symmetry given in Definition \ref{def:radial-symmetry} is readily adapted to the current setting of topological pairs (as already noted in Remark \ref{rem:self-similar_planar}).

In \cite{MannRafi}, Mann and Rafi introduced the following notion of a self-similar end space.

\begin{Def}
\( \cE \) is \emph{self-similar} if given pairwise-disjoint clopen subsets \( \cE_1, \ldots, \cE_n \) such that \( \cE = \cE_1 \sqcup \cdots \sqcup \cE_n \) there exists \( i \in \{1, \ldots, n \} \) and \( A \subset \cE_i \) such that \( A \) is open and homeomorphic to \( \cE \). 
\end{Def}

The goal is to prove:

\begin{Thm}\label{thm:SSandPS}
Let \(\cE\) denote the end space of an orientable surface \(S\). Then, \( \cE \) is self-similar if and only if \( \cE \) has radial symmetry.
\end{Thm}

Before continuing, we note that despite using the language of surfaces in our definition of \( \cE \), this section can all be framed in terms of second-countable Stone spaces and their groups of homeomorphisms (and hence, by Stone duality, everything can also be framed in terms of second-countable Boolean algebras). 

For Theorem~\ref{mainUncountable}, we will only require the forward direction; however, we believe this to be a useful characterization and so include a proof of the converse as well. 

We will make use of the following binary relation on \( \cE \) introduced in \cite{MannRafi}. 
Recall that given \( f \in \Homeo(S) \), we let \( \widehat f \) denote the induced homeomorphism in \( \Homeo(\cE) \).

\begin{Def}
Given two points \( x, y \in \cE \) we set \( y \leq x \) if, for every neighborhood \( U \) of \( x \), there exists a neighborhood \( V \) of \( y \) and \(f \in \Homeo(S)\) so that \( \widehat f(V) \subset U \). If \(  y \leq x \) and \( x \leq y \), we say that x and y are \emph{order equivalent}.
\end{Def}

Let \( [\cE] \) denote the set of order equivalence classes, then \( ([\cE], \leq) \) is a partially ordered set. 
We say a point in \( \cE \) is \emph{maximal} if its corresponding class in \( [\cE] \) is maximal with respect to the partial order and we let \(\mathcal{M} \) denote the set of maximal points in \( \cE \). 

At the end of the section in Proposition~\ref{prop:star-maximal}, we will see that if \( \cE \) is self-similar, then a point of \( \cE \) is maximal if and only if it is a star point.


\subsection{Self-similarity implies radial symmetry}

The following sequence of lemmas will be used to show that self-similarity implies radial symmetry. 

\begin{Lem}
\label{lem:nbhd}
If \( \cE \) is self-similar, then every open neighborhood of a maximal point of \( \cE \) contains a clopen subset homeomorphic to \( \cE \). 
\end{Lem}

\begin{proof}
Let \( x \in \cE \) be maximal and let \( U \) be an open neighborhood of \( x \).
By possibly shrinking \( U \), we may assume that \( U \) is clopen.  
Let \( V = \cE \ssm U \).
Using the maximality of \( x \), for each \( y \in V \) there exists an open neighborhood \( V_y \) of \( y \) such that \( U \) contains an open subset homeomorphic to \( V_y \).  
By possibly shrinking \( V_y \), we may assume that \( V_y \subset V \) and that \( V_y \) is clopen in \( \cE \). 

As \( V \) is clopen, it is compact; hence, there exists \(  y_1, \ldots, y_n \in V \) such that \( V=~\bigcup_{i=1}^n V_i \), where \( V_i = V_{y_i} \). 
Let \( V_1' = V_1 \), \( V_2' = V_2 \ssm V_1' \), \( V_3' = V_3 \ssm (V_1'\cup V_2') \), etc, so that 
we obtain \( V_1', \ldots, V_n' \) satisfying \[ \cE=U \sqcup V_1' \sqcup \cdots \sqcup V_m' \] with the property that \( U \) contains an open subset homeomorphic to \( V_i' \) for each \( i \in \{1, \ldots, n\} \). 
Now as \( \cE \) is self-similar, either \( U \) or one of the \( V_i' \) contains a clopen subset homeomorphic to \( \cE \).
In either case, we are done as \( U \) contains an open homeomorphic copy of \( V_i' \).
\end{proof}

\begin{Lem}
\label{lem:subset_maximal}
Let \( A \subset \cE \) be open and homeomorphic to \( \cE \).
If \( x \in A \) is maximal in \( A \), then \( x \) is also maximal in \( \cE \).
\end{Lem}

\begin{proof}
Let \( y \in \cE \) be comparable to \( x \), let \( U \) be an open neighborhood of \( x \), and let \( f \co \cE \to A \) be a homeomorphism. 
By the maximality of \( x \) in \( A \), there exists an open neighborhood \( V \) of \( f(y) \) in \( A \) such that \(U\) contains a homeomorphic copy of \( V \).
Hence, \( U \) contains an open neighborhood homeomorphic to \( f^{-1}(V) \), giving the desired open neighborhood of \( y \); therefore, \( x \) is maximal in \(\cE\). 
\end{proof}

\begin{Lem}
\label{lem:existence}
If \( \cE \) is self-similar, then there exists \( x \in \mathcal M \) such that every open neighborhood of \( x \) contains an open neighborhood of \( x \) homeomorphic to \( \cE \). 
\end{Lem}

\begin{proof}
Let \( \mathcal M \) denote the set of maximal points of \( \cE \). 
The proof splits into two cases, either \( |\mathcal M| =1 \) or \( |\mathcal M| > 1 \).

First assume that \( \mathcal M = \{x\} \).
Let \( U \) be an open neighborhood of \( x \).
By possibly shrinking \( U \), we may assume that \( U \) is clopen.
The self-similarity of \( \cE \) guarantees that either \( U \) or \( \cE \ssm U \) contains a subset \( A \) homeomorphic to \( \cE \).
The uniqueness of \( x \) together with Lemma~\ref{lem:subset_maximal} implies that \( A \) must contain \( x \) and hence \( A \subset U \).
Therefore, \( x \) satisfies the conclusion of the lemma. 

Now assume \( |\mathcal M| > 1 \).
It follows from Lemma~\ref{lem:nbhd} and Lemma~\ref{lem:subset_maximal} that every point of \( \mathcal M \) is an accumulation point of \( \mathcal M \) (in particular, \( \mathcal M \) is a Cantor space).
This also follows from \cite[Proposition 4.8]{MannRafi}.
For the rest of the argument, we fix a metric \( d \) on \( \cE \). 

Choose a maximal point \( x_1 \) of \( \cE \) and let \( U_1 \) be the open ball of radius 1 about \( x_1 \).
By Lemma~\ref{lem:nbhd}, there exists an open subset \( A_1 \) of \( \cE \) contained in \( U_1 \) and homeomorphic to \( \cE \). 
Now choose \( x_2 \) to be a maximal point of \( \cE \) contained in \( A_1 \), which can be done by Lemma~\ref{lem:subset_maximal}. 
Let \( U_2 \) be an open neighborhood of \( x_2 \) contained in the intersection of \( A_1 \) with the open ball of radius \( \frac12 \) about \( x_2 \). 
Again by Lemma~\ref{lem:nbhd}, we can choose an open subset \( A_2 \) homeomorphic to \( \cE \) and contained in \( U_2 \).
Proceeding in this fashion we build a nested sequence \( \{A_n\}_{n\in\bn} \) of open subsets each of which is homeomorphic to \( \cE \) and such that \( A_n \) has diameter less than \( \frac2n \). 

As \( \cE \) is compact and the \( A_n \) are nested, \( \bigcap_{n\in\bn} A_n \) is non-empty; moreover, by construction, the intersection must have diameter 0 and hence there exists \( x \in \cE \) such that \( \{x\} = \bigcap_{n\in\bn} A_n \).
Thus, \( x \) satisfies the conclusion of the lemma. 
\end{proof}

Let \( \widehat \bn = \bn \cup \{\infty\} \) be the one-point compactification of \( \bn \). 
Let \( x \) be the point of \( \cE \) given by the proof of Lemma~\ref{lem:existence} and define \( Z_x \) to be the space \( Z_x = \cE \times \widehat\bn / \sim \), where \( (x,n) \sim (x,m) \) for all \( n,m \in \bn \), \( (y,\infty) \sim (y', \infty) \) for all \( y, y' \in \cE \), and \( (x,n) \sim (y, \infty) \) for every \( n \in \bn \) and every \( y \in \cE \). 
Observe that \( Z_x \) is compact, Hausdroff, and admits a countable basis consisting of clopen subsets (note that both \( \cE \) and \( \widehat \bn \) have these properties). In the lemma below, we will show that \(Z_x\) is homeomorphic to \(\cE\) and then use this to conclude that \(\cE\) has radial symmetry. 

\begin{Lem}
\label{lem:Zx}
\( Z_x \) is homeomorphic to \( \cE \). 
\end{Lem}

\begin{proof}
Choose a nested sequence of clopen neighborhoods \( \{U_n\}_{n\in\bn} \) of \( x \) in \( \cE \) such that  \( U_1 = \cE \) and \( \{x\} = \bigcap_{n\in\bn} U_n \).
Let \( \widehat x \) denote the equivalence class of \((x,n)\) in \( Z_x \) and choose a nested sequence of clopen neighborhoods \( \{V_n\}_{n\in\bn} \) of \( \widehat x \) in \( Z_x \) such that \(V_1 = Z_x\) and \( \{\widehat x\} = \bigcap_{n\in\bn} V_n \).

Using a back-and-forth argument, we will construct a sequence of pairwise-disjoint clopen subsets \( \{W_n\}_{n\in\bn} \) of \( Z_x \), a sequence of pairwise-disjoint clopen subsets \( \{T_n\}_{n\in\bn} \) of \( \cE \), and homeomorphisms \( f_n\co W_n \to T_n \) such that 
\begin{enumerate}[(i)]
\item \(W_n \subset Z_x \ssm \{\widehat x\} \),
\item \( T_n \subset \cE\ssm\{x\} \),
\item \( V_n \ssm V_{n+1} \subset W_n \) and
\item \( U_n\ssm U_{n+1} \subset T_n \).
\end{enumerate}
Given such conditions, we see that \( Z_x \ssm \{x\} = \bigsqcup_{n\in\bn} W_n \) and \( \cE \ssm \{x\} = \bigsqcup_{n\in\bn} T_n\), so the map \( f \co Z_x \to \cE \) defined by \( f|_{W_n} = f_n \) and \( f(\widehat x) = x \) is well-defined.
Moreover, it is easy to see that \( f \) is bijective, maps convergent sequences to convergent sequences, and is therefore a continuous bijection between compact Hausdorff spaces; hence, \( f \) is a homeomorphism. 

One additional bit of notation: for each \( n \in \bn \), let \( \cE_n = \left(\cE\ssm\{x\}\right) \times \{n\} \subset Z_x \). 
Before beginning, note that for any open neighborhood \( V \) of \(  \widehat x  \) in \( Z_x \), there exists \( N_V \in \bn \) such that \( \cE_n \subset V \) for all \( n > N_V \): this can easily be seen by lifting  to \( \cE \times \widehat \bn \).

\begin{figure}[h]
\begin{center}
\begin{overpic}[trim = 1.25in 7in 1in 1.25in, clip=true, totalheight=0.32\textheight]{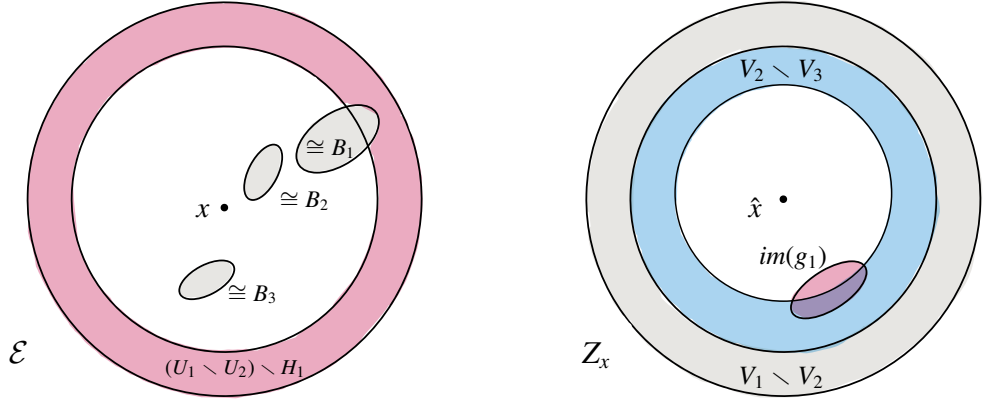}
\put(0,5){\large{$\cE$}}
\put(18,19){$x$}
\put(28.5,25){\footnotesize{$\cong B_1$}}
\put(26,20){\footnotesize{$\cong B_2$}}
\put(21,11){\footnotesize{$\cong B_3$}}
\put(15,4){\scriptsize{$(U_1 \ssm U_2) \ssm H_1$}}
\put(70,3){\small{$V_1 \ssm V_2$}}
\put(70,32.5){\small{$V_2 \ssm V_3$}}
\put(55,5){\large{$Z_x$}}
\put(71,19){$\hat x$}
\put(72,14.5){\small{$im(g_1)$}}
\end{overpic}
\caption{An example of the base case for recursion, where $V_1\ssm V_2 = B_ 1 \sqcup B_2 \sqcup B_3$.}\label{fig:base}
\end{center}
\end{figure}

We proceed by recursion. For the base case, let \( N = N_{V_2} \) and set \( B_k = (V_1\ssm V_2) \cap \cE_k \) for all \( k \in \{1, \ldots, N\} \).
Note, by the choice of \( N \), that \( V_1 \ssm V_2 = B_1 \sqcup \cdots \sqcup B_N \). 
We will now find an embedding \( h_1 \co V_1\ssm V_2 \hookrightarrow \cE \ssm \{x\} \) (see Figure~\ref{fig:base}). 
Let us identify \( B_1 \) with the corresponding subset of \( \cE \ssm \{x\} \). 
By compactness, there exists \( m_1 \) such that \( B_1 \subset U_1 \ssm U_{m_1} \).
By Lemma~\ref{lem:existence}, there exists an open neighborhood \( A \subset U_{m_1} \)  of \( x \) homeomorphic to \( \cE \).
We can now find an open subset of \( A \)---and hence \( U_{m_1} \)---homeomorphic to \( B_2 \).
Again by compactness, there exists \( m_2 \) such that \( B_2 \subset U_{m_1} \ssm U_{m_2} \). 
Continuing in this fashion, we construct a copy of \( V_1 \ssm V_2 = B_1 \sqcup \cdots \sqcup B_N \) in the set \( U_1 \ssm U_{m_N} \);
let \( h_1\co V_1\ssm V_2 \hookrightarrow \cE \ssm \{x\} \) denote the corresponding embedding.

Going back, let \( H_1 = \mathrm{image}(h_1) \subset \cE \) and let  \( g_1 \co (U_1\ssm U_2) \ssm H_1 \to \cE_{N+1} \subset Z_x \) be the restriction of the identification \( \cE \to \cE\times \{N+1\} \subset Z_x \).
Note that since \(V_1\ssm V_2 \subset \bigcup_{i=1}^N \cE_i\), the image of \(g_1\) is disjoint from \(V_1 \ssm V_2\). 
Finally, to finish the base case, let \( T_1 = (U_1\ssm U_2) \cup H_1 \) and let \( W_1 = (V_1 \ssm V_2 ) \cup \mathrm{image}(g_1) \), and let \( f_1 \co W_1 \to T_1 \) be the homeomorphism given by \( f_1= h_1 \sqcup g_1^{-1} \). 

Proceeding recursively, suppose that we have \( W_n \subset Z_x\ssm \{\widehat x\} \), \( T_n \subset \cE \ssm\{x\} \), and a homeomorphism \( f_n \co W_n \to T_n \) as desired.
We proceed exactly as in the base case with \( V_1 \) replaced by \( V_{n+1} \), \( V_2 \) replaced by \( V_{n+2} \cup W_1 \cup \cdots \cup W_n \), and \( U_1 \) replaced by an open copy of \( \cE \) contained in the clopen set \( U_{n+1} \ssm (T_1 \cup \cdots \cup T_n) \). 
\end{proof}

Since  \[ Z_x \ssm \{\widehat x\} = \bigsqcup_{n\in\bn} [(\cE\ssm\{x\})\times \{n\}], \] the space $Z_x$ has radial symmetry; hence, Lemma \ref{lem:Zx} establishes that \( \cE \) has radial symmetry, allowing us to conclude that self-similarity implies radial symmetry:

\begin{Lem}
\label{lem:forward}
If \( \cE \) is self-similar, then \( \cE \) has radial symmetry. \qed
\end{Lem}

\begin{Rem}\label{rem:rsE}
Note that \( [(\cE\ssm\{x\})\times \{n\}] \cong \cE \ssm \{x\}\) so that, by Lemma~\ref{lem:Zx}, \(\cE\ssm\{x\} =  \bigsqcup_{n\in\bn} \cE_n\) where \(\cE_n \cup \{x\}\cong \cE\) for all \(n \in \bn\).
\end{Rem}


\subsection{Radial symmetry implies self-similarity}

Next, we establish that radial symmetry implies self-similarity.
To do so, we will use the following proposition, which is a consequence of Propositions 4.7 and 4.8 of \cite{MannRafi}. 

\begin{Prop}[Mann--Rafi]
\label{thm:mr}
Let \( \cE \) be realizable as the end space of an orientable surface in which every compact subsurface is displaceable.
If \( [\cE] \) has a unique maximal element and the set of maximal points \(\mathcal{M}\) of \( \cE \) is either a singleton or infinite, then \( \cE \) is self-similar. \qed
\end{Prop}

We proceed with a sequence of lemmas to see that if \( \cE \) has radial symmetry, then it meets the hypotheses of Proposition~\ref{thm:mr} and is, therefore, self-similar. 

\begin{Lem}
\label{lem:intersection}
Given \( x \in \cE \), let \( \bar {\mathcal  O}_x \) denote the closure of the orbit of \( x \) under the action of \( \Homeo(\cE) \).
If \[ X_\cE = \bigcap_{x\in \cE} \bar{\mathcal O}_x \neq \varnothing, \]
then \( [\cE] \) has a unique maximal element and every point in \( X_\cE \) is maximal.
\end{Lem}

\begin{proof}
Let \( x \in X_\cE \), let \( y \in \cE \ssm \{x\} \), and let \( U \) be a neighborhood of \( x \).
By assumption, \( x \in \bar{\mathcal O}_y \) and hence there exists \( f \in \Homeo(S) \) such that \( \widehat f(y) \subset U \), which guarantees \( y \leq x \).
As \( y \) was arbitrary we can conclude that \( y \leq x \) for all \( y \in \cE \), which establishes that \( x \) is maximal and shows that \( [\cE] \) has a unique maximal element.
We have also established that every point in \( X_\cE \) is maximal.
\end{proof}

\begin{Lem}
\label{lem:pointed}
If \( \cE \) has radial symmetry and \( x \) is a star point, then \( x\in X_\cE \).
\end{Lem}

\begin{proof}
Let \( x \) be a star point of \( \cE \) and write \( \cE \ssm \{x\} = \bigsqcup_{n\in\bn} \cE_n \) with the \( \cE_n \) as in the definition of radial symmetry. 
Let \( y \in \cE \ssm \{x\} \) and, without loss of generality, assume \( y \in \cE_1 \). 
For \( n > 1 \), choose \( f_n \in \Homeo(S) \) such that \( \widehat f_n(y) \in \cE_n \).
The sequence \( \{\widehat f_n(y)\}_{n\in\bn} \) is discrete in \( \cE \ssm \{x\} \) and hence must converge to \( x \). 
It follows that  \( x \in \bar{\mathcal{O}_y} \), and since \(y\) was arbitrary, \( x \in X_\cE \).
\end{proof}

\begin{Lem}
\label{lem:displaceable}
Let \( S_\cE \) be either a planar or an orientable infinite-genus 2-manifold such that \( \cE(S_\cE) = \cE \).
Every compact subsurface of \( S_\cE \) is displaceable. 
\end{Lem}

\begin{proof}
If \( |\cE|=1 \) and \( S_\cE \) is planar, then \( S_\cE \) is homeomorphic to \( \br^2 \) and the result is clear. 
Let \( K \subset S_\cE \) be a compact subsurface and let \( g_K \) denote the genus of \( K \).  
By possibly enlarging \( K \), we can assume that each component of \( \partial K \) is separating and each component of \( S_\cE \ssm K \) is unbounded.
Let \( b_1, \ldots, b_n \) denote the components of \( \partial K \) and, for each \( i \leq n \),  let \( U_i \) be the component of \( S_\cE \ssm K \) whose boundary is \( b_i \). 
It follows that \( \cE = \widehat U_1 \sqcup \widehat U_2 \sqcup\cdots\sqcup \widehat U_n \). 

Let \( x \) be a star point of \( \cE \) and write \( \cE \ssm \{x\} = \bigsqcup_{n\in\bn} \cE_n \) with the \( \cE_n \) as in the definition of radial symmetry. 
Given a finite permutation \( \sigma \in \rm{Sym}(\bn) \), choose \( f_\sigma \in \Homeo(S) \) satisfying \( \widehat f_\sigma(x) = x \) and \( \widehat f_\sigma(\cE_n) = \cE_{\sigma(n)} \). 

Without loss of generality, assume \( x \in \widehat U_n \). 
It follows that there exists \( M \in \bn \) such that \( \cE_m \subset \widehat U_n \) for all \( m > M \). 
Choose a finite permutation \( \sigma \in \rm{Sym}(\bn) \) such that \( \sigma(j) > M \) for all \( j \in \{1, \ldots, M \} \). 
Now, since \( \widehat f_\sigma(\widehat U_i) \subset \widehat U_n \) for each \( i < n \), we can choose \( b_1', b_2', \ldots, b_{n-1}' \) to be pairwise-disjoint separating simple closed curves in \( U_n \) such that, for each \( i < n \), there is a component \( V_i \) of \( S_\cE \ssm b_i' \) such that \( \widehat V_i = \widehat f_\sigma(\widehat U_i) \). 
It is left to choose \( b_n' \): choose any separating simple closed curve contained in \( U_n \) that co-bounds a genus-\( g_K \) subsurface \( K ' \) with \( b_1', \ldots, b_{n-1}' \). 
By construction, \( K \) and \( K' \) are disjoint, homeomorphic, and their complementary components induce homeomorphic partitions of \( \cE \); hence, by the classification of surfaces, there exists \( f \in \Homeo(S_\cE) \) such that \( f(K) = K' \). 
As \( K \) was arbitrary, every compact subsurface of \( S_\cE \) is displaceable. 
\end{proof}

\begin{Lem}
\label{lem:converse}
If \( \cE \) has radial symmetry, then \( \cE \) is self-similar.
\end{Lem}

\begin{proof}
We intend to appeal to Proposition \ref{thm:mr}, so we first note that, by Lemma \ref{lem:displaceable}, \( \cE \) is realizable as the end space of an orientable surface in which every compact subsurface is displaceable. 
By Lemma~\ref{lem:pointed}, \( X_\cE \) is nonempty and hence, by Lemma~\ref{lem:intersection}, \( [\cE] \) has a unique maximal element. 
Now suppose that \( \cE \) has at least two maximal points.
Let \( x \) be a star point and write \( \cE \ssm\{x\} = \bigsqcup_{n\in\bn} \cE_n \) with the \( \cE_n \) as in the definition of radial symmetry.
If \( y \) is another maximal point, then there exists \( i \in \bn \) such that \( y \in \cE_i \). 
But for every \( n \in \bn \) there exists \( f_n \in \Homeo(S) \) such that \( \widehat f_n(\cE_i) = \cE_n \) and hence there exists infinitely many maximal points.
We have established that \( [\cE] \) has a unique maximal element and that  \( \mathcal M \) is either a singleton or infinite, therefore, by Proposition~\ref{thm:mr}, \( \cE \) is self-similar. 
\end{proof}

Together, Lemma \ref{lem:forward} and Lemma \ref{lem:converse} establish Theorem~\ref{thm:SSandPS}.

We end this section with observing the correspondence between star points and maximal points in end spaces that have radial symmetry (or equivalently, are self-similar).

\begin{Prop}
\label{prop:star-maximal}
Let \( \cE \) denote the end space of an orientable surface \( S \).
If \( \cE \) has radial symmetry, then a point of \( \cE \) is maximal if and only if \( \cE \) is a star point.
\end{Prop}

\begin{proof}
Lemma~\ref{lem:intersection} and Lemma~\ref{lem:pointed} imply that every star point is maximal, so we only need to concern ourselves with the forward direction.
By the assumption of radial symmetry, \( \cE \) contains a star point \( x \), which---as just noted---is necessarily maximal.
Now let \( y \) be another maximal point.
By Lemma~\ref{lem:displaceable}, we may assume that every compact subsurface of \( S \) is displaceable. 
By Lemma~\ref{lem:converse}, \( \cE \) is self-similar and so we can apply \cite[Remark 4.15 and Lemma 4.17]{MannRafi} to see that there exists a homeomorphism \( f\co S \to S \) such that \( f(x) = y \).
In particular, \( y \) is star point.
\end{proof}

\section{The proof of Theorem~\ref{mainUncountable}}
\label{sec:trichotomy}

We begin this section by coalescing results from \cite{MannRafi} to establish a partition of all orientable surfaces into three categories, described below in Theorem \ref{thm:partialclass}.
Afterwards, using the equivalence of radial symmetry and self-similarity (Theorem \ref{thm:SSandPS}), we assemble the results from Section \ref{Allcock} and Section \ref{Denumerable} to establish our main theorem, Theorem \ref{mainUncountable}.

\begin{Thm}\label{thm:partialclass} If $S$ is an orientable 2-manifold in which every compact subsurface is displaceable, then either the end space of \(S\) is self-similar or it is doubly pointed.

\end{Thm}

\begin{proof}
Let \( \cE = \cE(S) \). 
Using Proposition 4.7 along with the construction in Example~2.5  from \cite{MannRafi}, we conclude that since all compact surfaces in $S$ are displaceable, either $[\cE]$ has a unique maximal element---whose equivalence class consists of a singleton, two points, or a Cantor set---or it has two maximal elements---each of which is a singleton. Assuming every compact subsurface is displaceable, by Proposition 2.8 from \cite{MannRafi}, the end space is self-similar if and only if $\mathcal{M}$ is a singleton or a Cantor set consisting of points of the same type. Thus, if all compact surfaces in $S$ are displaceable and $\cE$ is not self-similar, then either there is one maximal element---consisting of two points---or two maximal elements---each of which is a singleton. 

In either case, there are exactly two maximal points $x_1$ and $x_2$, and so each has a finite $\Homeo(S)$-orbit.
Now observe that the orbit of every other point is infinite: indeed, if \( x \in \cE \) is not maximal, then, up to relabelling, we can assume \( x \leq x_1 \); therefore, for every neighborhood \( U \) of \( x_1 \) there exists \( f \in \Homeo(S) \) such that \( f(x) \in U \ssm \{x_1\} \) and so the orbit of \( x \) is infinite.
Thus, $\cE$ is doubly pointed. 
\end{proof}

Using this theorem, we see that Theorem~\ref{mainUncountable} covers all orientable infinite-type 2-manifolds with no planar ends. Combining Theorems~\ref{thm:lack},~\ref{thm:pointed},~\ref{thm:doubly-pointed},~\ref{thm:SSandPS}, and~\ref{thm:partialclass} gives the main theorem of the paper.

\begin{Thm} \label{mainUncountable} 
Let \( S \) be an orientable infinite-genus 2-manifold with no planar ends and let $G$ be an arbitrary group. Then: 
\begin{enumerate}
\item If the end space of $S$ is self-similar, there exists a  complete hyperbolic metric on $S$ whose isometry group is $G$ if and only if $G$ is countable. 

\item If the end space $S$ is doubly pointed, then the isometry group of any complete hyperbolic metric on $S$ is virtually cyclic. 

\item If $S$ contains a compact non-displaceable subsurface, then there exists a  complete hyperbolic metric on $S$ whose isometry group is $G$ if and only if $G$ is finite. 

\end{enumerate}
Moreover, every such 2-manifold satisfies at least one of the above hypotheses. \qed
\end{Thm}


\section{Regular covers (proof of Theorem~\ref{thm:covering})}\label{sec:covers}

In this section, we give a strengthening of Proposition~\ref{prop:finite-cover} under the additional assumption that the surface has a self-similar end space, establishing Theorem~\ref{thm:covering}.

\begin{Prop}\label{prop:sscovering}
Let \( S \) be an orientable infinite-genus 2-manifold with no planar ends and self-similar end space.
If  \( G \) is any countable group, then there exists a regular covering \( \pi \co S \to S \) such that the deck group associated to \( \pi \) is isomorphic to \( G \).
\end{Prop}

\begin{proof}
In the case that \( G \) is finite, we refer to Proposition~\ref{prop:finite-cover}, so let us assume that \( G \) is infinite.
By Theorem \ref{thm:SSandPS}, we know that the end space \( \cE \) of \( S \) has radial symmetry and, moreover, by Remark~\ref{rem:rsE}, \( \cE \) has a star point \( x \) such that \( \cE \ssm \{x\} = \bigsqcup_{n\in\bn} E_n \) with \( E_n \cup \{x\} \) homeomorphic to \( \cE \).
Let \( Y_S^G \) and \( R \) be the surfaces constructed in Section~\ref{sec:radial}, so that \( Y_S^G \) is homeomorphic to \( S \) (Lemma~\ref{lem:Y-homeo}) and \( \cE(R) \) is homeomorphic to \( \cE \). 
From the construction of \( Y_S^G \), it is not hard to see that the quotient of \( Y_S^G \) by the action of \( G \) is homeomorphic to the quotient of \( R \) obtained by identifying the boundary components of \( R \) in pairs via orientation-reversing homeomorphisms; hence, \( G \backslash Y_S^G \) is an infinite-genus surface with no planar ends whose end space is homeomorphic to \( \cE \).
Therefore, the associated covering is the desired one. 
\end{proof}

Proposition~\ref{prop:finite-cover} and \ref{prop:sscovering} immediately give: 

\begin{Thm}\label{thm:covering}
Let \( S \) be an orientable infinite-genus 2-manifold with no planar ends.
If  \( G \) is any finite group, then there exists a regular covering \( \pi \co S \to S \) such that the deck group associated to \( \pi \) is isomorphic to \( G \). If, additionally, the end space of \(S\) is self-similar and \( G \) is any countable group, then there exists a regular covering \( \pi \co S \to S \) such that the deck group associated to \( \pi \) is isomorphic to \( G \). \qed
\end{Thm}

\section{Application to (pure) mapping class groups}\label{sec:app2}

We proceed with an application to mapping class groups.
The \emph{extended mapping class group of a surface \( S \)} is the group  \( \mcg^\pm(S) \) of homotopy classes of homeomorphisms \( S \to S \); the \emph{mapping class group} \( \mcg(S) \) is the subgroup of \( \mcg^\pm(S) \) of orientation-preserving classes; and, the \emph{pure mapping class group}, denoted by \( \pmcg(S) \), is the kernel of the action of \( \mcg(S) \) on \( \cE(S) \).
Recently, there have been a number of results uncovering the algebraic structure of (pure) mapping class groups of infinite-type surfaces.  
One notable example is the algebraic rigidity of (pure) mapping class groups of infinite-type surfaces \cite{BDR}, that is, given two orientable infinite-type 2-manifolds \( S \) and \( S' \),  if \( \mcg(S) \) (or \( \pmcg(S) \)) and \( \mcg(S') \) (resp. \( \pmcg(S') \)) are isomorphic, then \( S \) and \( S' \) are homeomorphic. 

With the exception of a finite number of well-understood cases, algebraic rigidity can be deduced for (pure) mapping class groups of orientable finite-type 2-manifolds from the work of Ivanov \cite{Ivanov}, Korkmaz \cite{Korkmaz}, and Luo \cite{Luo}.
In the finite-type case, combining results of Birman--Lubotsky--McCarthy \cite{BLM} (computing the algebraic rank) and Harer \cite{Harer} (computing the virtual cohomological dimension), it is, moreover, possible to determine the topology of \( S \) from algebraic invariants of either \( \mcg(S) \) or \( \pmcg(S) \). 

This leads to the following problem:

\begin{problem}{Provide a list of algebraic invariants of  \( \mcg(S) \) and/or \( \pmcg(S) \) that determine the topology of \( S \).}
\end{problem}

Combining several recent results with Theorem~\ref{thm:finite}, we provide such a list for a countable family of infinite-type surfaces.
Before stating the result, we recall several definitions.
A group is \emph{residually finite} if the intersection of all its normal subgroups is trivial. 
A group is \emph{left cyclically orderable} if it admits a cyclic order invariant under left multiplication.
We avoid the definition of a cyclic order by recalling two facts:
(1) every subgroup of a left cyclically orderable group is itself left cyclically orderable and 
(2) \( \Homeo(\bS^1) \) is left cyclically orderable, where \( \bS^1 \) denotes the circle (see \cite{Zheleva}).
In fact, a countable group is left cyclically orderable if and only if it can be realized as a subgroup of \( \Homeo(\bS^1) \) (see  \cite[Theorem 2.2.14]{CalegariCircular}).
From this, one readily deduces that every finite left cyclically orderable group is cyclic.

\begin{Thm}\label{thm:rigid}
Let \( \Omega_n \) be an \( n \)-ended orientable infinite-genus 2-manifold with no planar ends and let \( G = \pmcg(S) \) for some orientable 2-manifold \( S \).
Then, \( S \) is homeomorphic to \( \Omega_n \) if and only if  \( G  \) satisfies each of the following conditions:
\begin{enumerate}[(i)]
\item \( G \) is not residually finite,
\item \( G \) is not left cyclically orderable,
\item \( \mathrm{Hom}(G, \bz) \) has algebraic rank \( n-1 \), and
\item \( G \) is finite index in \( \mathrm{Aut}(G) \).
\end{enumerate}
\end{Thm}

\begin{proof}
Let \( \cE \) be the end space of \( S \).
By Patel--Vlamis \cite{PatelAlgebraic}, \( \pmcg(S) \) is residually finite if and only if \( S \) has finite genus.
Hence, (i) implies \( S \) has infinite genus.
For \( \pmcg(\Omega_n) \), \cite{PatelAlgebraic} implies (i) holds.

By Bavard--Walker \cite{BavardTwo}, if \( S \) has an isolated planar end, then \( \pmcg(S) \) is left cyclically orderable.
Hence, (ii) implies \( S \) has no isolated planar ends.
For \( \Omega_n \), let \( H \) be a non-cyclic finite group. Then, by Lemma \ref{lem:action-trivial} and Theorem~\ref{thm:finite}, there exists a hyperbolic surface \( X \) homeomorphic \( \Omega_n \) such that \( \isom(X) \cong H \) and, moreover, every element of \( \isom(X) \) is orientation-preserving and fixes every end of \( \Omega_n \).
It is well known that no two distinct isometries of a hyperbolic surface are homotopic; hence, \( \isom(X) \) can be realized as a subgroup of \( \pmcg(\Omega_n) \); in particular, \( \pmcg(\Omega_n) \) is not left cyclically orderable, establishing (ii).

By Aramayona--Patel--Vlamis \cite{AramayonaFirst}, the assumption that \( \mathrm{Hom}(\pmcg(S), \bz) \) has algebraic rank \( n-1  \) together with the fact \( S \) has infinite genus---from condition (i)---guarantees \( S \) has  \( n \) non-planar ends.
For \( \Omega_n \), the same result \cite{AramayonaFirst} tells us that (iii) holds for \( \pmcg(\Omega_n) \).

As \( \pmcg(S) \) is the kernel of the action of \( \mcg(S) \) on \( \cE \), we see that \( \pmcg(S) \) is finite index in \( \mcg(S) \) if and only if \( \cE \) is finite. 
By Bavard--Dowdall--Rafi \cite{BDR}, \( \mathrm{Aut}(\pmcg(S)) \) is isomorphic to \( \mcg^\pm(S) \); hence, for \( G \) to be finite index in \( \mathrm{Aut}(G) \), the cardinality of \( \cE \) must be finite.
For \( \Omega_n \)---which is \( n \)-ended---we see (iv) holds.

We have established that \( \pmcg(\Omega_n) \) satisfies all four conditions of the theorem.
Now, since the cardinality of \( \cE \) is finite and \( S \) has no isolated planar ends, it follows that \( S \) has no planar ends.
But, we know that \( S \) has \( n \) non-planar ends by condition (iii); hence, by the classification of surfaces, \( S \) is homeomorphic to \( \Omega_n \). 
\end{proof}


\section{Notes on planar ends and on finite genus}\label{sec:other surfaces}

We end with some observations about surfaces with finite genus or a finite number of planar ends, the goal of which is to justify the assumptions in the main theorems.

First, let us recall that the isometry group of any compact genus-\( g \) hyperbolic surface is finite with order at most \( 168(g-1) \) when \( g \geq 2 \).
We first observe that if we remove the compactness restriction, but focus on a fixed genus, there are obstructions for the action of finite groups:

\begin{Prop}\label{prop:finitegenus}
Let \( S \) be an orientable  2-manifold with finite genus \( g \) and non-abelian fundamental group.
Let \( G \) be the isometry group of a complete constant-curvature Riemannian metric on \( S \).
\begin{itemize}
\item If \( g = 0 \) and \( G \) is finite, then \( G \) is isomorphic to a subgroup of \( O(3) \), the 3-dimensional orthogonal group.
\item If \( g = 1 \), then \( G \) is a finite quotient of a 2-dimensional crystallographic group (and hence admits a generating set consisting of at most four elements).
\item If \( g \geq 2 \), then \( |G| \leq 168(g-1) \). 
\end{itemize}
\end{Prop}

\begin{proof}
By the assumption on the fundamental group of \( S \), we have that \( G \) is the isometry group of a complete hyperbolic metric on \( S \).
First, we observe that \( G \) is finite: if \( g=0 \) then this is by assumption; otherwise, \( g > 0 \) and every compact genus-\( g \) subsurface of \( S \) is non-displaceable; hence, by Lemma \ref{lem:finite-isom}, \( G \) is finite.

Let \( \cE \) denote the end space of \( S \) and let \( S_g \) be an orientable closed genus-\( g \) surface.
Fix an embedding of \( \cE \) into \( S_g \). Then, by the classification of surfaces, \( S \) is homeomorphic to \( S_g \ssm \cE \).
Moreover, every homeomorphism \( S \to S \) has a unique extension to a homeomorphism \( S_g \to S_g \);
in particular, we can view \( G \) as a subgroup of \( \Homeo(S_g) \). 
It is well known (for instance, see \cite[Theorem 2.8]{Epstein}) that every periodic homeomorphism of a 2-manifold can be realized as an isometry of some Riemannian metric (of constant curvature) on the manifold.
Therefore, every periodic homeomorphism of a 2-manifold is smooth. 
This allows us to choose any Riemannian metric \( \rho \) on \( S_g \) and define 
\[ 
\rho'_G = \sum_{g\in G} g^* \rho.
\]
Then \( \rho'_G \) is a Riemannian metric on \( S_g \) and, by the uniformization theorem, there exists a complete Riemannian metric of constant curvature \( \rho_G \) on \( S_g \) conformally equivalent to \( \rho'_G \).
By construction, \( G < \isom( S_g, \rho_G) \).

If \( g = 0 \), then \( \rho_G \) is isometric to the round metric on the 2-sphere and hence we can realize \( G \) as a subgroup of \( O(3) \).
If \( g = 1 \), then \( G \) must be the quotient of a 2-dimensional crystallographic group: there are 17 such groups, each of which is generated by at most four elements (see \cite{Hillman} for a survey). 
If \( g\geq 2 \), then \( \rho_G \) is a complete hyperbolic metric and, by Hurwitz's theorem on automorphisms, \( |G| \leq 168(g-1) \). 
\end{proof}

We note that the finiteness assumption in the genus-0 case is necessary: it is possible for a complete hyperbolic planar surface to have an infinite isometry group.
This follows from the fact the sphere with a Cantor set removed is a regular cover of every closed genus-\( g \) surface with \( g \geq 2 \). For instance, let \( \gamma \in \pi_1(S_g) \) be nontrivial and corresponding to a simple closed curve on \( S_g \). Let \( N_\gamma \) denote the normal closure of the group generated by the conjugates of \(\gamma\). Then the regular cover of \( S_g \) associated to \( N_\gamma \) is homeomorphic to a sphere with a Cantor set removed---in fact, all such planar covers are obtained in a similar fashion \cite{Maskit}.

Note that if a 2-manifold has a non-zero finite number of planar ends, then it contains a compact non-displaceable surface; hence, by Lemma \ref{lem:finite-isom}, any isometry group of a complete hyperbolic metric is finite.
To end, we show that there exists a finite group that cannot be realized as an isometry group of a complete hyperbolic metric on each surface of this type.

\begin{Prop}\label{prop:finiteplanar}
If \( S \) is an orientable 2-manifold with a non-zero finite number of planar ends, then there exists a finite group that cannot be realized as the isometry group of a complete hyperbolic metric on \( S \).
\end{Prop}

\begin{proof}
Let \( n \) be the number of planar ends of \( S \).
Let \( G \) be a non-abelian simple finite group of order greater than \( n! \) and suppose \( G \) is the isometry group of some hyperbolic structure on \( S \).
We first observe that \( G \) must act on the planar ends of \( S \) trivially: indeed, the action of \( G \) on the planar ends of \( S \) induces a homomorphism from \( G \) to the symmetric group on \( n \) letters, but the assumption on the cardinality of \( G \) implies that the kernel is non-empty and hence, by the simplicity of \( G \), must be all of \( G \).
Using again the work of Bavard--Walker \cite{BavardTwo}, we must have that \( G \) is cyclically orderable and thus cyclic, which contradicts our assumption that \( G \) is non-abelian.
\end{proof}

\bibliographystyle{amsplain}
\bibliography{Bib-realization}

\end{document}